\documentclass[11pt,english]{amsart}
\usepackage[T1]{fontenc}
\usepackage[latin9]{inputenc}
\usepackage{amsthm}
\usepackage{amstext}
\usepackage{amssymb}
\usepackage{wasysym}
\usepackage{esint}

\makeatletter
\usepackage{amsmath,amsfonts,amssymb,amsthm,epsfig}

\usepackage{color}
\usepackage[final,allcolors=blue,colorlinks=true]{hyperref}

\def\R{\mathbb R}
\def\N{\mathbb N}
\def\C{\mathbb C}

\def\X{\mathbb X}

\def\al{\alpha}
\def\be{\beta}
\def\ga{\gamma}
\def\de{\delta}
\def\ep{\epsilon}
\def\la{\lambda}

\def\ta{\theta}
\def\var{\varphi}

\def\na{\nabla}
\def\De{\Delta}      

\def\cal{\mathcal}
\def\L{\mathcal L}                                       
\def\wq{\infty}
\def\pa{\partial}

\def\Id{{\rm Id}\,}

\def\rad{\text{\rm rad}}

\newcommand{\D}{{\rm d}}

\newcommand{\medint}{-\kern -,375cm\int}         
\newcommand{\medintinrigo}{-\kern -,315cm\int}
\newcommand{\wto}{\rightharpoonup}                
\newcommand\esssup{\text{\rm \,esssup\,}}

\numberwithin{equation}{section}
\newtheorem{theorem}{Theorem}[section]

\newtheorem{lemma}[theorem]{Lemma}

\newtheorem{proposition}[theorem]{Proposition}
\newtheorem{remark}[theorem]{Remark}

\theoremstyle{definition}

\makeatother

\usepackage{babel}
\begin{document}
\title[Uniqueness of Ground State of Choquard Equation]{Uniqueness and Nondegeneracy of Ground States for Choquard Equations in three dimensions}

                                \author[C.-L. Xiang]{Chang-Lin Xiang}

\address[]{University of Jyvaskyla, Department of Mathematics and Statistics, P.O. Box 35, FI-40014 University of Jyvaskyla, Finland.}
\email[]{Xiang\_math@126.com}

\begin{abstract}
We obtain   uniqueness and nondegeneracy results for ground states of  Choquard equations $-\Delta u+u=\left(|x|^{-1}\ast|u|^{p}\right)|u|^{p-2}u$ in $\mathbb{R}^{3}$, provided that  $p>2$ and $p$ is sufficiently close to 2.
\end{abstract}

\maketitle

{\small
\smallskip
\newline
\tableofcontents}
\bigskip

\section{Introduction and main results}

\subsection{Introduction}

In this paper, we study the nonlinear elliptic problem
\begin{eqnarray}
-\Delta u+\la u=\left(|x|^{-1}\ast|u|{}^{p}\right)|u|^{p-2}u &  & \text{in }\R^{3},\label{eq: Choquard equation}
\end{eqnarray}
where $\la>0$, $1<p<\wq$ are positive  constants, $\De=\sum_{i=1}^{3}\pa_{x_{i}x_{i}}$
is the usual Laplacian operator in $\R^{3}$.

Equation (\ref{eq: Choquard equation}) is usually called the nonlinear
Choquard or Choquard-Pekar equation. It is closely related to the
focusing time-dependent Choquard equation
\begin{eqnarray}
i\psi_{t}=-\De\psi-\left(|x|^{-1}\ast|\psi|{}^{p}\right)|\psi|^{p-2}\psi &  & \text{in }\R^{3}\times\R_{+}.\label{eq:  time-dependent Hartree equation}
\end{eqnarray}
This can be seen from the fact that the function $\psi(x,t)=e^{i\la t}u(x)$
gives a solitary wave for equation (\ref{eq:  time-dependent Hartree equation})
whenever $u$ solves equation (\ref{eq: Choquard equation}). In this
context, equation (\ref{eq: Choquard equation}) is known as the stationary
nonlinear Choquard equation. In the case $p=2$, equation (\ref{eq: Choquard equation})
is reduced to
\begin{eqnarray}
-\Delta u+\la u=\left(|x|^{-1}\ast|u|{}^{2}\right)u &  & \text{in }\R^{3}.\label{eq: Hartree equ.}
\end{eqnarray}
Equation (\ref{eq: Hartree equ.}) is also called the nonlinear Hartree
or Schr\"odinger-Newton equation. It was used to describe the quantum
mechanics of a polaron at rest in the work of Pekar \cite{Pekar-1954}.
It was also used by Choquard to describe an electron trapped in its
own hole in a certain approximating to Hartree-Fock theory of one
component plasma in 1976, see e.g. Lieb \cite{Lieb-1976}. For more
mathematical and physics background for problems (\ref{eq: Choquard equation})-(\ref{eq: Hartree equ.}),
we refer the readers to e.g. \cite{Gross-Book,Lieb-1976,Lieb-Simon-1977,Lions-1980,Moroz-Penrose-Tod-1998}
and the references therein.

In this paper, we study ground state solutions (see below) of equation
(\ref{eq: Choquard equation}). By a solution to equation (\ref{eq: Choquard equation}),
we mean a function $u\in H^{1}(\R^{3})\cap L^{6p/5}(\R^{3})$ such
that for any function $\var$ belonging to $C_{0}^{\wq}(\R^{3})$,
the infinitely differentiable functions in $\R^{3}$ with compact
support, there holds
\[
\int_{\R^{3}}\left(\na u\cdot\na\var+\la u\var-\left(|x|^{-1}\ast|u|{}^{p}\right)|u|^{p-2}u\var\right)\D y=0.
\]
The solution is well defined due to the Hardy-Littlewood-Sobolev inequality
\begin{eqnarray*}
\iint_{\R^{3}\times\R^{3}}\frac{|v(x)|^{p}|v(y)|^{p}}{|x-y|}\D x\D y\le A\left(\int_{\R^{3}}|v|^{\frac{6p}{5}}\D x\right)^{\frac{5}{6p}} &  & \text{for }v\in L^{\frac{6p}{5}}(\R^{3}),
\end{eqnarray*}
where $A>0$ is a constant independent of $v\in L^{6p/5}(\R^{3})$.
We are concerned about the uniqueness and the so called nondegeneracy
(see below) of ground state solutions of problem (\ref{eq: Choquard equation}).
Before giving our results, let us first summarize some known results
about ground state solutions of problem (\ref{eq: Choquard equation}).

It is well known \cite{Lions-1980,Lions-concentration-comp-I,Moroz-Van Schaftingen-2013}
that equation (\ref{eq: Choquard equation}) is variational. So solutions
to equation (\ref{eq: Choquard equation}) can be found by investigating
critical points of related variational functionals. For instance,
Moroz and Van Schaftingen \cite{Moroz-Van Schaftingen-2013} proved
the existence of positive radial solutions to equation (\ref{eq: Choquard equation})
in $H^{1}(\R^{3})$ by exploring minimizers of the variational problem
\[
\inf\left\{ \frac{\int_{\R^{3}}|\na u|^{2}+|u|^{2}\D y}{\left(\int_{\R^{3}}\left(|x|^{-1}\ast|u|^{p}\right)|u|^{p}\D y\right)^{1/p}}:u\in H^{1}(\R^{3}),u\ne0\right\} .
\]
In fact, in the same way, Moroz and Van Schaftingen \cite{Moroz-Van Schaftingen-2013}
studied much more general problems than (\ref{eq: Choquard equation}).
From a physical point of view \cite{Cazenave-Lions-1982,Lenzmann-2009,Lieb-1976,Lions-1980,Lions-concentration-comp-I},
the most interesting critical points are minimizers of the problem
\begin{equation}
m(N,p)=\inf\left\{ E_{p}(u):u\in{\cal A}_{N}\right\} ,\label{prob: minimizing}
\end{equation}
where $E_{p}:H^{1}(\R^{3})\to\R$ is an energy functional defined
by
\[
E_{p}(u)=\frac{1}{2}\int_{\R^{3}}|\na u|^{2}\D x-\frac{1}{2p}\iint_{\R^{3}\times\R^{3}}\frac{|u(x)|^{p}|u(y)|^{p}}{|x-y|}\D x\D y,
\]
and ${\cal A}_{N}$ is an admissible set given by
\[
{\cal A}_{N}=\left\{ u\in H^{1}(\R^{3}):\|u\|_{2}=N\right\}
\]
for a give number $N>0$. Here $\|u\|_{2}^{2}=\int_{\R^{3}}|u|^{2}\D y$
denotes the norm of the space $L^{2}(\R^{3})$. Following the convention
of Cazenave and Lions \cite{Cazenave-Lions-1982} (see also \cite{Frank-Lenzmann-2013,Frank-Lenzmann-Silvestre-2013,Lenzmann-2009}),
we call any minimizer $Q$ of problem (\ref{prob: minimizing}) a
\emph{ground state solution, }or simply \emph{ground state},\emph{
}of problem (\ref{eq: Choquard equation}) in ${\cal A}_{N}$. We
summarize the existence result of ground states of problem (\ref{eq: Choquard equation})
along with a list of basic properties as follows.

\begin{theorem}\label{thm: properties of GS} Assume that $5/3<p<7/3$.
Then for any given number $N$, $N>0$, the following results hold.

(1) (Existence) There exits at least one ground state for problem
(\ref{eq: Choquard equation}) in ${\cal A}_{N}$.

(2) (Symmetry) For any ground state $Q\in{\cal A}_{N}$ of problem
(\ref{eq: Choquard equation}), there exists a strictly decreasing
positive function $v:[0,\wq)\to(0,,\wq)$ such that $Q=v(|\cdot-y|)$
for a point $y\in\R^{3}$.

(3) (Regularity) Let $Q\in{\cal A}_{N}$ be an arbitrary ground state
of problem (\ref{eq: Choquard equation}). Then $Q$ solves equation
(\ref{eq: Choquard equation}) with $\la$ being a positive Lagrange
multiplier. Moreover, $Q\in W^{2,s}(\R^{3})\cap C^{\wq}(\R^{3})$
holds for any $s>1$.

(4) (Decay) For any radial ground state $Q\in{\cal A}_{N}$ of problem
(\ref{eq: Choquard equation}) with $2\le p<7/3$, there exists a
constant $\ga>0$ such that $Q(x)=O(e^{-\ga|x|})$ holds for $|x|$
sufficiently large. \end{theorem}

For a complete proof of Theorem \ref{thm: properties of GS}, we refer
the readers to e.g. Moroz and Van Schaftingen \cite{Moroz-Van Schaftingen-2013}.
See also Lions \cite{Lions-concentration-comp-I} for the existence
of ground states for problem (\ref{eq: Hartree equ.}). For the sake
of completeness, we give a sketch of the proof of Theorem \ref{thm: properties of GS}
in Appendix \ref{sec: Existence}.

\subsection{Main results}

In this paper, we are concerned about the uniqueness and the so called
nondegeneracy (see below) of ground states of problem (\ref{eq: Choquard equation}).
The motivation comes from the well known fact that the uniqueness
and nondegeneracy of ground states $Q$ of problem (\ref{eq: Choquard equation})
play a fundamental role in the stability and blow up analysis for
the corresponding solitary wave solutions $\psi(x,t)=e^{i\la t}Q(x)$
of the focusing time-dependent Hartree equation (\ref{eq:  time-dependent Hartree equation}),
see e.g. Lenzmann \cite{Lenzmann-2009} and the references therein.
We also refer the interested readers to e.g. \cite{Chang et al-2007,Kwong1989,Weinstein-1985}
for studies on the uniqueness and nondegeneracy of ground states for
nonlinear Schr\"odinger equations with local nonlinearities, and
to e.g. \cite{Fall-Valdinoci-2014,Frank-Lenzmann-2013,Frank-Lenzmann-Silvestre-2013,Lenzmann-2009}
for studies on the same topics for nonlocal problems.

However, in striking contrast to the questions of existence, it seems
fair to say that extremely little is known about uniqueness and nondegeneracy
of ground states for problem (\ref{eq: Choquard equation}), except
in the isolated case $p=2$, which is due to Lieb \cite{Lieb-1976}
and Lenzmann \cite{Lenzmann-2009} respectively. The purpose of this
paper is to provide some results in this respect. Our first result
reads as follows.

\begin{theorem}\label{thm: Main theorem 1-Uniqueness} There exists
a number $0<\de<1/3$ such that for any $p$, $2<p<2+\de$, and for
any $N>0$, there exists a unique ground state $Q\in H^{1}(\R^{3})$
for problem (\ref{eq: Choquard equation}) with $\|Q\|_{2}=N$ up
to translations. In particular, there exists a unique positive radial
ground state $Q=Q(|x|)>0$ for problem (\ref{eq: Choquard equation})
with $\|Q\|_{2}=N$. \end{theorem}

As already mentioned, Theorem \ref{thm: Main theorem 1-Uniqueness}
is also true for $p=2$, which is due to Lieb \cite{Lieb-1976}.

Recall that ground states of problem (\ref{eq: Choquard equation})
are minimizers of problem (\ref{prob: minimizing}). However, note
that the functional $E_{p}$ in problem (\ref{prob: minimizing})
is not convex. So the conventional way to prove uniqueness for minimizers
does not work. On the other hand, since ground states are positive
radial solutions to equation (\ref{eq: Choquard equation}) in $H^{1}(\R^{3})$,
a natural idea to derive uniqueness for ground states of problem (\ref{eq: Choquard equation})
is to show that equation (\ref{eq: Choquard equation}) admits a unique
positive radial solution in $H^{1}(\R^{3})$. Indeed, as one of his
main results, Lieb \cite{Lieb-1976} proved the quite strong result
that in the case $p=2$ equation (\ref{eq: Choquard equation}), namely
equation (\ref{eq: Hartree equ.}), admits a unique positive radial
solution in $H^{1}(\R^{3})$. One may try to extend his arguments
to derive uniqueness for positive radial solutions to equation (\ref{eq: Choquard equation})
in the general case $p\ne2$. Unfortunately, this does not seem to
work. The arguments of Lieb depend heavily on the particular nonlinearity
of equation (\ref{eq: Hartree equ.}). In the general case when $p\neq2$,
the strong nonlinearity of the term $\left(|x|^{-1}\ast|u|{}^{p}\right)|u|^{p-2}u$
in equation (\ref{eq: Choquard equation}) prevents one from using
the arguments of Lieb \cite{Lieb-1976}. For details of the arguments
of Lieb \cite{Lieb-1976}, we refer to e.g. Lieb \cite{Lieb-1976}
or Lenzmann \cite[Appendix A]{Lenzmann-2009}. Therefore new ideas
to derive uniqueness of ground states for problem (\ref{eq: Choquard equation})
are in need.

Inspired by the recent works \cite{Fall-Valdinoci-2014,Frank-Lenzmann-2013,Frank-Lenzmann-Silvestre-2013,Lenzmann-2009},
in the present paper we will apply a combination of compactness argument
and local uniqueness argument to prove Theorem \ref{thm: Main theorem 1-Uniqueness}.
Let $N>0$ be given and let $Q$ be the unique positive radial solution
to equation (\ref{eq: Hartree equ.}) in $H^{1}(\R^{3})$ with $\|Q\|_{2}=N$.
First we prove a compactness result (see Theorem \ref{thm: compactness result}
in Section \ref{sec:Compactness-result}), which states that positive
radial ground states for problem (\ref{eq: Choquard equation}) converges
to $Q$ as $p$ tends to 2. Then we derive a local uniqueness result
(see Proposition \ref{prop: Implicit theorem to local uniqueness}
in Section \ref{sec:Proof-of-main results}), which states that equation
(\ref{eq: Choquard equation}) has a unique positive radial solution
in a sufficiently small neighborhood of $Q$ when $p>2$ and $p$
is sufficiently close to 2. Finally, arguing by contradiction, we
conclude Theorem \ref{thm: Main theorem 1-Uniqueness}. To prove the
compactness result, apriori estimates for the corresponding Lagrange
multipliers are in need, which requires a careful analysis on the
equation (\ref{eq: Choquard equation}) and the problem (\ref{prob: minimizing}).
We will also deduce an uniform estimate for the sequence of positive
radial ground states of problem (\ref{eq: Choquard equation}). To
derive the local uniqueness result, we need a deeper knowledge on
the unique positive radial ground state $Q$ to equation (\ref{eq: Hartree equ.}).
Precisely, we need to know that the linearized operator for equation
(\ref{eq: Hartree equ.}) associated to $Q$ is nondegenerate. Fortunately,
this fact has been confirmed by Lenzmann \cite[Theorem 1.4]{Lenzmann-2009}.
Then we apply an implicit function argument to derive the local uniqueness
result.

Before proceeding further, we would like to remark a recent progress
on the strong uniqueness result of Lieb \cite{Lieb-1976}. By applying
the moving plane method of Chen et al. \cite{Chen-Li-Ou-2006}, Ma
and Zhao \cite{Ma-Zhao-2010} proved that every positive solution
to equation (\ref{eq: Hartree equ.}) in $H^{1}(\R^{3})$ is radially
symmetric about some point in $\R^{3}$. Thus in view of the uniqueness
result of Lieb \cite{Lieb-1976}, Ma and Zhao \cite{Ma-Zhao-2010}
concluded that the Hartree equation (\ref{eq: Hartree equ.}) admits
a unique positive solution in $H^{1}(\R^{3})$ up to translations.
We refer the readers to Ma and Zhao \cite{Ma-Zhao-2010} for more
symmetry results on positive solutions to nonlinear equations.

Now we move to our second result in this paper. Let $\de>0$ be defined
as in Theorem \ref{thm: Main theorem 1-Uniqueness}. Let $Q=Q(|x|)>0$
be the unique positive radial ground state for problem (\ref{eq: Choquard equation})
with $\|Q\|_{2}=N$ and $2<p<2+\de$. We define the linear operator
$\L_{+,p}$ associated to $Q$ by
\begin{equation}
\L_{+,p}\xi=-\De\xi+\la\xi-(p-1)\left(|x|^{-1}\ast Q^{p}\right)Q^{p-2}\xi-p\left(|x|^{-1}\ast\left(Q^{p-1}\xi\right)\right)Q^{p-1}\label{eq: linearized operator}
\end{equation}
acting on $L^{2}(\R^{3})$ with domain $H^{1}(\R^{3})$. Following
the idea of Lenzmann \cite{Lenzmann-2009}, we obtain the nondegeneracy
for ground states of problem (\ref{eq: Choquard equation}).

\begin{theorem}\label{thm: Main result 2-Nondegeneracy} Let $\de>0$
be defined as in Theorem \ref{thm: Main theorem 1-Uniqueness} and
$2<p<2+\de$. Consider the unique positive radial ground state $Q$
for problem (\ref{eq: Choquard equation}) with $\|Q\|_{2}=N$. Then
there exists a number $0<\de^{\prime}\le\de$ such that for all $p$,
$2<p<2+\de^{\prime}$, the operator $\L_{+,p}$ defined as in (\ref{eq: linearized operator})
is nondegenerate. That is,
\[
\text{{\rm Ker}}\L_{+,p}=\text{{\rm span}}\left\{ \pa_{x_{1}}Q,\pa_{x_{2}}Q,\pa_{x_{3}}Q\right\} .
\]
\end{theorem}

In the case $p=2$, Theorem \ref{thm: Main result 2-Nondegeneracy}
is due to Lenzmann \cite{Lenzmann-2009}. We will prove Theorem \ref{thm: Main result 2-Nondegeneracy}
by following the argument of Lenzmann \cite{Lenzmann-2009}. Let us
now give some remarks on Theorem \ref{thm: Main result 2-Nondegeneracy}
before we close this section.

\begin{remark} (1) Let $Q$, $\la$ and $p$ satisfy the assumptions
in Theorem \ref{thm: Main result 2-Nondegeneracy}. Consider problem
(\ref{eq: Choquard equation}) in the complex valued Sobolev space
$H^{1}(\R^{3};\C)$. Then the linearized operator $\L$ at $Q$ is
given by
\[
\begin{aligned}\L\xi & =-\De\xi+\la\xi-\left(|x|^{-1}\ast Q^{p}\right)Q^{p-2}\xi-\frac{p-2}{2}\left(|x|^{-1}\ast Q^{p}\right)Q^{p-2}\left(\xi+\bar{\xi}\right)\\
 & \qquad-\frac{p}{2}\left(|x|^{-1}\ast\left(Q^{p-1}\left(\xi+\bar{\xi}\right)\right)\right)Q^{p-1}.
\end{aligned}
\]
However, note that $\L$ is not $\C$-linear. To study the kernel
of $\L$, we view $\L$ as a combination of operators $\L_{+,p}$
and $\L_{-,p}$, which act on the real part $\text{{\rm Re}}\xi$
and the imaginary part $\text{{\rm Im}}\xi$ of $\xi$ respectively.
That is,
\[
\L\xi=\L_{+,p}\text{{\rm Re}}\xi+i\L_{-,p}\text{{\rm Im}}\xi.
\]
Here $\L_{+,p}$ is the linear operator defined as in (\ref{eq: linearized operator}),
and $\L_{-,p}$ is defined as
\[
\L_{-,p}=-\De+\la-\left(|x|^{-1}\ast Q^{p}\right)Q^{p-2}.
\]
Since the ground state $Q$ does not change sign in the whole space
$\R^{3}$, it is easy to see that
\[
\text{{\rm Ker}}\L_{-,p}=\text{{\rm span}}\{Q\}
\]
holds. Hence, by Theorem \ref{thm: Main result 2-Nondegeneracy},
we obtain that
\[
\text{{\rm Ker}}\L=\left\{ \sum_{k=1}^{3}a_{k}\pa_{x_{k}}Q+ibQ:a_{1},a_{2},a_{3},b\in\R\right\} .
\]

(2) An immediate application of Theorem \ref{thm: Main result 2-Nondegeneracy}
that is of importance in the stability and blowup analysis of solitary
waves for the focusing time-dependent Hartree equation (\ref{eq:  time-dependent Hartree equation})
is given in terms of a coercivity estimate of $\L_{+}$. Precisely,
denote by $\phi$ the first eigenfunction of $\L_{+}$ acting on $L^{2}(\R^{3})$
and set $M=\text{{\rm span}}\left\{ \phi,\pa_{x_{1}}Q,\pa_{x_{2}}Q,\pa_{x_{3}}Q\right\} \subset L^{2}(\R^{3})$.
Then we can use Theorem \ref{thm: Main result 2-Nondegeneracy} to
derive the lower bound
\begin{eqnarray*}
\langle\L_{+}\eta,\eta\rangle\ge c\|\eta\|_{H^{1}(\R^{3})}^{2} &  & \text{for }\eta\in M^{\bot},
\end{eqnarray*}
where $c>0$ is a constant independent of $\eta$.\end{remark}

The rest of the paper is organized as follows. Section \ref{sec:Compactness-result}
is devoted to the compactness result of ground states for problem
(\ref{prob: minimizing}) when $p$ tends to 2. Section \ref{sec:Proof-of-main results}
is devoted to the proofs of Theorem \ref{thm: Main theorem 1-Uniqueness}
and Theorem \ref{thm: Main result 2-Nondegeneracy}. For the sake
of completeness, we briefly prove Theorem \ref{thm: properties of GS}
in Appendix \ref{sec: Existence}. We also give in Appendix  \ref{sec: Regularity-of F}
a short proof of the regularity of the functional $F$ used in Section
\ref{sec:Proof-of-main results}.

Our notations are standard. We denote by $B_{r}(0)$ the ball centered
at the origin in $\R^{3}$ with radius $r$. For $1\le q\le\wq$,
we use $L^{q}(\R^{3})$ to denote the Banach space of Lebesgue measurable
functions $u$ such that the norm
\[
\|u\|_{q}=\begin{cases}
\left(\int_{\R^{3}}|u|^{q}\D x\right)^{\frac{1}{q}} & \text{if }1\le q<\infty\\
\esssup_{\R^{3}}|u| & \text{if }q=\infty
\end{cases}
\]
is finite. A function $u$ belongs to the Sobolev space $H^{1}(\R^{3})$
if and only if $u\in L^{2}(\R^{3})$ and its first order weak partial
derivatives also belong to $L^{2}(\R^{3})$. We equip $H^{1}(\R^{3})$
with the norm
\[
\|u\|_{H^{1}(\R^{3})}=\|u\|_{2}+\|\na u\|_{2}.
\]
We denote by $L_{\rad}^{q}(\R^{3})$ and $H_{\rad}^{1}(\R^{3})$ the
subspaces of radial functions in $L^{q}(\R^{3})$ and $H^{1}(\R^{3})$
respectively. By the usual abuse of notation, we write $f(r)=f(x)$
with $r=|x|$ whenever $f$ is a radial function in $\R^{n}$.

\section{Compactness analysis\label{sec:Compactness-result}}

In this section, our aim is to prove the following compactness result.

\begin{theorem}\label{thm: compactness result} Let $\{p_{n}\}\subset(2,7/3)$
be an arbitrary sequence with $\lim_{n\to\wq}p_{n}=2$ and let $N>0$
be given. Let $Q_{p_{n}}=Q_{p_{n}}(|x|)>0$ be a positive radial ground
state for problem (\ref{eq: Choquard equation}) with $p=p_{n}$ and
$\|Q_{p_{n}}\|_{2}=N$ for all $n\in\N$. Then we have that
\begin{eqnarray*}
Q_{p_{n}}\to Q_{2} &  & \text{in }H^{1}(\R^{3}).
\end{eqnarray*}
Here $Q_{2}=Q_{2}(|x|)>0$ is the unique positive radial ground state
for problem (\ref{eq: Hartree equ.}) with $\|Q_{2}\|_{2}=N$. \end{theorem}

We divide the proof of Theorem \ref{thm: compactness result} into
several lemmas. For simplicity, we introduce the notations
\begin{eqnarray*}
K(u)=\frac{1}{2}\int_{\R^{3}}|\na u|^{2}\D x & \text{and} & D_{p}(u)=\frac{1}{2p}\iint_{\R^{3}\times\R^{3}}\frac{|u(x)|^{p}|u(y)|^{p}}{|x-y|}\D x\D y.
\end{eqnarray*}
Then the energy functional $E_{p}$ can be written as
\[
E_{p}(u)=K(u)-D_{p}(u).
\]

\subsection{Apriori estimates for Lagrange multipliers.}

In this subsection, we give an apriori estimate for Lagrange multipliers.
First we have an equivalent variational characterization for the constrained
problem (\ref{prob: minimizing}).

\begin{lemma} \label{lem: equivalent minimum}Assume that $5/3<p<7/3$.
Then for any $N>0$, we have
\begin{equation}
m(N,p)=-C_{1}(p)\sup\left\{ \left(\frac{D_{p}(u)^{2}}{K(u)^{3p-5}}\right)^{\frac{1}{7-3p}}:u\in{\cal A}_{N}\right\} ,\label{eq: equivalent minimum}
\end{equation}
where $C_{1}(p)$ is a positive constant given by
\begin{equation}
C_{1}(p)=\frac{7-3p}{3p-5}\left(\frac{3p-5}{2}\right)^{\frac{2}{7-3p}}.\label{eq: C-1-p}
\end{equation}
 \end{lemma}
\begin{proof}
For any $u\in{\cal A}_{N}$, we have $t^{3/2}u(tx)\in{\cal A}_{N}$
for all $t>0$. An elementary calculation gives that
\[
\inf_{t>0}E_{p}\left(t^{3/2}u(tx)\right)=-C_{1}(p)\left(\frac{D_{p}(u)^{2}}{K(u)^{3p-5}}\right)^{\frac{1}{7-3p}}
\]
with $C_{1}(p)$ given by (\ref{eq: C-1-p}). Then Lemma \ref{lem: equivalent minimum}
follows from above easily.
\end{proof}
Lemma \ref{lem: equivalent minimum} can also be found in Moroz and
Van Schaftingen \cite{Moroz-Van Schaftingen-2013} in a more general
context. It is easy to infer from (\ref{eq: equivalent minimum})
that if $Q\in{\cal A}_{N}$ is a minimizer of problem (\ref{prob: minimizing}),
then
\[
m(N,p)=E_{p}(Q)=-C_{1}(p)\left(\frac{D_{p}(Q)^{2}}{K(Q)^{3p-5}}\right)^{\frac{1}{7-3p}}.
\]
Next we have the following observation.

\begin{lemma}\label{lem: N is 1} Assume that $5/3<p<7/3$. Then
for any $N>0$, we have
\[
m(N,p)=m(1,p)N^{\frac{10-2p}{7-3p}}.
\]
\end{lemma}
\begin{proof}
Note that $u\in{\cal A}_{1}$ if and only if $u_{N}\equiv N^{5/2}u(Nx)\in{\cal A}_{N}$.
An elementary calculation gives that
\[
\left(\frac{D_{p}(u_{N})^{2}}{K(u_{N})^{3p-5}}\right)^{\frac{1}{7-3p}}=\left(\frac{D_{p}(u)^{2}}{K(u)^{3p-5}}\right)^{\frac{1}{7-3p}}N^{\frac{10-2p}{7-3p}}.
\]
Thus we deduce from (\ref{eq: equivalent minimum}) and above equality
that
\[
\begin{aligned}m(N,p) & =-C_{1}(p)\sup\left\{ \left(\frac{D_{p}(u_{N})^{2}}{K(u_{N})^{3p-5}}\right)^{\frac{1}{7-3p}}:u\in{\cal A}_{1}\right\} \\
 & =-N^{\frac{10-2p}{7-3p}}C_{1}(p)\sup\left\{ \left(\frac{D_{p}(u)^{2}}{K(u)^{3p-5}}\right)^{\frac{1}{7-3p}}:u\in{\cal A}_{1}\right\} \\
 & =m(1,p)N^{\frac{10-2p}{7-3p}}.
\end{aligned}
\]
 The proof of Lemma \ref{lem: N is 1} is complete.
\end{proof}
Lemma \ref{lem: N is 1} implies that it is sufficient to prove Theorem
\ref{thm: compactness result} in the case $N=1$. In the rest of
this paper, we shall only consider the case $N=1$. For simplicity
we write
\[
m(p)=m(1,p).
\]
Then (\ref{eq: equivalent minimum}) gives that
\begin{equation}
m(p)=-C_{1}(p)\sup\left\{ \left(\frac{D_{p}(u)^{2}}{K(u)^{3p-5}}\right)^{\frac{1}{7-3p}}:u\in{\cal A}_{1}\right\} \label{prob: equivalent minimum with N is 1}
\end{equation}
with $C_{1}(p)$ given by (\ref{eq: C-1-p}).

We will need a lower bound for $m(p)$. Recall that by the classical
Hardy-Littlewood-Sobolev inequality (see e.g. \cite{Lieb-Loss-book}),
there exists an absolute constant $A>0$ such that, for any $5/3<p<5$,
we have that
\begin{eqnarray}
\iint_{\R^{3}\times\R^{3}}\frac{|u(x)|^{p}|u(y)|^{p}}{|x-y|}\D x\D y\le A\|u\|_{\frac{6p}{5}}^{2p} &  & \forall\, u\in L^{\frac{6p}{5}}(\R^{3}).\label{eq: H-L-S ineq.}
\end{eqnarray}
Then by interpolation inequality, we have that
\begin{eqnarray*}
\|u\|_{\frac{6p}{5}}\le\|u\|_{2}^{\theta}\|u\|_{6}^{1-\theta} &  & \forall\, u\in L^{2}(\R^{3})\cap L^{6}(\R^{3}),
\end{eqnarray*}
with $0<\ta=(5-p)/2p<1$. By Sobolev inequality, there exists an absolute
constant $B>0$ such that
\begin{eqnarray}
\|u\|_{6}\le B\|\na u\|_{2} &  & \forall\, u\in H^{1}(\R^{3}).\label{eq: Sobolev ineq.}
\end{eqnarray}
Hence combining above three inequalities we obtain that
\begin{equation}
\iint_{\R^{3}\times\R^{3}}\frac{|u(x)|^{p}|u(y)|^{p}}{|x-y|}\D x\D y\le AB^{3p-5}\|u\|_{2}^{5-p}\|\na u\|_{2}^{3p-5}\label{eq: Gagliado type ineq.}
\end{equation}
for all $u\in H^{1}(\R^{3})$. In particular, we deduce from (\ref{eq: Gagliado type ineq.})
that
\begin{eqnarray}
\iint_{\R^{3}\times\R^{3}}\frac{|u(x)|^{p}|u(y)|^{p}}{|x-y|}\D x\D y\le AB^{3p-5}\|\na u\|_{2}^{3p-5} &  & \forall\, u\in{\cal A}_{1}.\label{eq: Interpolation inequ.}
\end{eqnarray}
Hence combining (\ref{prob: equivalent minimum with N is 1}) and
(\ref{eq: Interpolation inequ.}) yields the following apriori estimate
for $m(p)$.

\begin{lemma} \label{lem: lower bound of m-p}Assume that $5/3<p<7/3$.
Then we have
\begin{equation}
m(p)\ge-C_{2}(p),\label{eq: lower bd of m}
\end{equation}
where $C_{2}(p)$ is a positive constant given by
\begin{equation}
C_{2}(p)=\frac{7-3p}{3p-5}\left(\frac{(3p-5)A(\sqrt{2}B)^{3p-5}}{4p}\right)^{\frac{2}{7-3p}},\label{eq: C-2-p}
\end{equation}
and $A$, $B$ are the positive absolute constants given by (\ref{eq: H-L-S ineq.})
and (\ref{eq: Sobolev ineq.}) respectively. \end{lemma}

Now let us consider Lagrange multipliers associated to minimizers
of problem (\ref{prob: minimizing}).

\begin{lemma} \label{lem: Lagrange multiplier is unique} Assume
that $5/3<p<7/3$. Let $Q$ be an arbitrary minimizer of problem (\ref{prob: minimizing})
with $N=1$. Consider the Lagrange multiplier $\la_{p}$ corresponding
to $Q$ such that equation (\ref{eq: Choquard equation}) is satisfied
by $Q$ with $\la=\la_{p}$. Then we have
\begin{equation}
\la_{p}=-\frac{2(5-p)}{7-3p}m(p).\label{eq: m-p to lambda-p}
\end{equation}
\end{lemma}
\begin{proof}
The proof is based on a Pohozaev type identity for solutions to equation
(\ref{eq: Choquard equation}). Note that $Q$ satisfies
\begin{eqnarray}
-\Delta Q-\left(|x|^{-1}\ast Q{}^{p}\right)Q^{p-1}=-\la_{p}Q &  & \text{in }\R^{3}.\label{eq: 2.star-1}
\end{eqnarray}
Multiplying each side of equation (\ref{eq: 2.star-1}) by $x\cdot\na Q$,
we obtain by integrating by parts that
\[
K(Q)-5D_{p}(Q)=-\frac{3}{2}\la_{p}\int_{\R^{3}}|Q|^{2}\D x.
\]
For details of the proof of above identity, we refer to Moroz and
Van Schaftingen \cite[Proposition 3.1]{Moroz-Van Schaftingen-2013}.
Multiplying each side of equation (\ref{eq: 2.star-1}) by $Q$, we
obtain that
\[
2K(Q)-2pD_{p}(Q)=-\la_{p}\int_{\R^{3}}|Q|^{2}\D x.
\]
Recall that $\|Q\|_{2}=1$. Combining above two identities yields
that
\begin{eqnarray}
K(Q)=\frac{\left(3p-5\right)\la_{p}}{2(5-p)} & \text{and} & D_{p}(Q)=\frac{\la_{p}}{5-p}.\label{eq: Pohozaev identity}
\end{eqnarray}
On the other hand, (\ref{prob: equivalent minimum with N is 1}) gives
us that
\[
m(p)=-C_{1}(p)\left(\frac{D_{p}(Q)^{2}}{K(Q)^{3p-5}}\right)^{\frac{1}{7-3p}},
\]
where $C_{1}(p)$ is defined as in (\ref{eq: C-1-p}). Hence we derive
from (\ref{eq: Pohozaev identity}) and above equation that
\[
m(p)=-\frac{7-3p}{2(5-p)}\la_{p}.
\]
This proves (\ref{eq: m-p to lambda-p}). The proof of Lemma \ref{lem: Lagrange multiplier is unique}
is complete now.
\end{proof}
We remark that Lemma \ref{lem: Lagrange multiplier is unique} implies
that the Lagrange multiplier $\la_{p}$ is independent of the choice
of minimizers of problem (\ref{prob: equivalent minimum with N is 1})
(with $N=1$). Furthermore, we have the following apriori estimates
for $\la_{p}$.

\begin{lemma} \label{lem: Lagrange multiplier is bounded } Let $\la_{p}$
be defined as in (\ref{eq: m-p to lambda-p}) for $5/3<p<7/3$. Then
for any compact subset $K\subset(5/3,7/3)$, we have
\[
0<\inf_{p\in K}\la_{p}\le\sup_{p\in K}\la_{p}<\wq.
\]
 \end{lemma}
\begin{proof}
By (\ref{eq: m-p to lambda-p}), it is equivalent to prove that for
any compact subset $K\subset(5/3,7/3)$, there holds
\[
-\wq<\inf_{p\in K}m(p)\le\sup_{p\in K}m(p)<0.
\]

The lower bound $\inf_{K}m>-\wq$ follows from (\ref{eq: lower bd of m})
of Lemma \ref{lem: lower bound of m-p}, since the function $p\mapsto C_{2}(p)$
is a positive continuous function for $5/3<p<7/3$.

So it remains to prove that $\sup_{p\in K}m(p)<0$. By(\ref{prob: equivalent minimum with N is 1}),
it is easy to see that $m(p)<0$ for all $5/3<p<7/3$. Hence we always
have $\sup_{p\in K}m(p)\le0$. To obtain the strict inequality, we
claim that the function $p\mapsto m(p)$ is upper semicontinuous.
Indeed, note the fact that for any $u\in H^{1}(\R^{3})$ fixed, the
function $p\mapsto D_{p}(u)$ is continuous. This fact can be proved
by the same argument as that of Lieb and Loss \cite[Section 8.14]{Lieb-Loss-book}.
We omit the details. Thus the function $p\mapsto E_{p}(u)$ is continuous
for any $u\in H^{1}(\R^{3})$ fixed. Now we can conclude that the
function $p\mapsto m(p)$ is upper semicontinuous, since $m(p)$ is
the infimum of a family of continuous functions $p\mapsto E_{p}(u)$
with $u\in{\cal A}_{1}$. Then we known that for any compact subset
$K\subset(5/3,7/3)$, there exists $p_{0}\in K$ such that
\[
\sup_{p\in K}m(p)=m(p_{0})<0.
\]
The proof of Lemma \ref{lem: Lagrange multiplier is bounded } is
complete.
\end{proof}

\subsection{Uniform estimates for ground states.}

Let $\{p_{n}\}\subset(2,7/3)$ be an arbitrary sequence such that
$\lim_{n\to\wq}p_{n}=2$. Let $Q_{p_{n}}=Q_{p_{n}}(|x|)>0$ be a positive
radial ground state for problem (\ref{eq: Choquard equation}) with
$p=p_{n}$ and $\|Q_{p_{n}}\|_{2}=1$ for all $n\in\N$. Then $Q_{p_{n}}$
satisfies
\begin{eqnarray}
-\Delta Q_{p_{n}}-\left(|x|^{-1}\ast Q_{p_{n}}^{p_{n}}\right)Q_{p_{n}}^{p_{n}-1}=-\la_{p_{n}}Q_{p_{n}} &  & \text{in }\R^{3},\label{eq: 2.star-2}
\end{eqnarray}
where $\la_{p_{n}}$ is defined as in (\ref{eq: m-p to lambda-p}).
In this subsection, we prove the following uniform estimate for the
sequence $\{Q_{p_{n}}\}$.

\begin{proposition} \label{prop:  Uniform estimates} There exists
a positive function $F\in L^{1}(\R^{3})\cap L^{\wq}(\R^{3})$, such
that
\begin{eqnarray}
Q_{p_{n}}\le F &  & \text{in }\R^{3}\label{eq: domination}
\end{eqnarray}
holds for all $n\in\N$. \end{proposition}

It has been shown by Theorem \ref{thm: properties of GS} that each
$Q_{p_{n}}$ is bounded and decays exponentially to zero at infinity.
However, we can not find a literature where a uniform estimate for
$Q_{p_{n}}$ with respect to $n$ is given. Hence we derive Proposition
\ref{prop:  Uniform estimates} to gives a uniform estimate for all
$Q_{p_{n}}$. We remark that our estimates are only precise enough
for use and far from optimal.

To prove Proposition \ref{prop:  Uniform estimates}, first we derive
the following uniform boundedness estimate for $Q_{p_{n}}$.

\begin{lemma}\label{lem: uniform bound of Q-p-n} There exists a
constant $M>0$ such that
\[
\sup_{n}\|Q_{p_{n}}\|_{\wq}\le M<\wq.
\]
\end{lemma}

In the rest of this section, we denote, for all $n\in\N$,
\[
V_{n}=|x|^{-1}\ast Q_{p_{n}}^{p_{n}}.
\]

\begin{proof}[Proof of Lemma \ref{lem: uniform bound of Q-p-n}]To
prove Lemma \ref{lem: uniform bound of Q-p-n}, we need the following
estimate
\begin{equation}
\sup_{n}\|V_{n}\|_{\wq}<\wq.\label{esti: uniform bound of V-n}
\end{equation}
Note that $-\De V_{n}=4\pi Q_{p_{n}}^{p_{n}}$ and that $V_{n}$ is
positive, radial and decreasing with respect to $|x|$. Hence $0\le V_{n}(x)\le V_{n}(0)$.
We only need to show that
\begin{equation}
\sup_{n}V_{n}(0)<\wq.\label{eq: equiv. unif. bound of V-n}
\end{equation}
By definition, we have that
\[
V_{n}(0)=\int_{\R^{3}}|x|^{-1}Q_{p_{n}}^{p_{n}}(x)\D x=\int_{B_{1}(0)}|x|^{-1}Q_{p_{n}}^{p_{n}}(x)\D x+\int_{\R^{3}\backslash B_{1}(0)}|x|^{-1}Q_{p_{n}}^{p_{n}}(x)\D x.
\]
Since $2<p_{n}<7/3$, there exists $r>1$ such that $r^{\prime}=r/(r-1)<3$
and $p_{n}r\le6$. Then H\"older's inequality gives that
\[
\int_{B_{1}(0)}|x|^{-1}Q_{p_{n}}^{p_{n}}(x)\D x\le\left(\int_{B_{1}(0)}|x|^{-r^{\prime}}\D x\right)^{1/r^{\prime}}\left(\int_{B_{1}(0)}Q_{p_{n}}^{p_{n}r}\D x\right)^{1/r}.
\]
Note that $Q_{p_{n}}\in H^{1}(\R^{3})$ is uniformly bounded. Thus
Sobolev inequality implies that $\int_{B_{1}(0)}Q_{p_{n}}^{p_{n}r}\D x\le C$
holds for a constant $C>0$ independent of $n$. Therefore
\[
\sup_{n}\int_{B_{1}(0)}|x|^{-1}Q_{p_{n}}^{p_{n}}(x)\D x<\wq.
\]
On the other hand, $2<p_{n}<7/3$ implies that $2<4p_{n}/3<6$. Combining
H\"older's inequality and Sobolev inequality yields that
\[
\int_{\R^{3}\backslash B_{1}(0)}|x|^{-1}Q_{p_{n}}^{p_{n}}(x)\D x\le\left(\int_{\R^{3}\backslash B_{1}(0)}|x|^{-4}\D x\right)^{\frac{1}{4}}\left(\int_{\R^{3}\backslash B_{1}(0)}Q_{p_{n}}^{\frac{4p_{n}}{3}}\D x\right)^{\frac{3}{4}}\le C^{\prime}<\wq
\]
for a constant $C^{\prime}>0$ independent of $n$. Combining above
two estimates completes the proof of (\ref{eq: equiv. unif. bound of V-n}).
Thus (\ref{esti: uniform bound of V-n}) holds.

Now we prove Lemma \ref{lem: uniform bound of Q-p-n} as follows.
Note that $Q_{p_{n}}$ satisfies
\begin{equation}
-\De Q_{p_{n}}+Q_{p_{n}}=(1-\la_{p_{n}})Q_{p_{n}}+V_{n}Q_{p_{n}}^{p_{n}-1}.\label{eq: 2.star-2.1}
\end{equation}
Since $2<p_{n}<7/3$, we have $9/2<6/(p_{n}-1)<6$. Since $p_{n}$
tends to $2$ as $n\to\wq$, $\{p_{n}\}$ is contained in a compact
subset of $(5/3,7/3)$. Therefore, by Lemma \ref{lem: Lagrange multiplier is bounded },
$\la_{p_{n}}$ is bounded uniformly for all $n\in\N$. Thus we easily
deduce that
\[
\sup_{n}\left\Vert (1-\la_{p_{n}})Q_{p_{n}}\right\Vert _{9/2}\le M_{1}<\wq
\]
for a constant $M_{1}>0$, and that
\[
\sup_{n}\|Q_{p_{n}}^{p_{n}-1}\|_{\frac{9}{2}}\le\sup_{n}\|Q_{p_{n}}\|_{2}^{(p_{n}-1)\theta}\|Q_{p_{n}}\|_{6}^{(p_{n}-1)(1-\ta)}\le M_{2}<\wq
\]
for a constant $M_{2}>0$, since $Q_{p_{n}}\in H^{1}(\R^{3})$ is
uniformly bounded, where $\theta\in(0,1)$ is a constant. Thus we
deduce from (\ref{esti: uniform bound of V-n}) and above estimate
that
\[
\sup_{n}\left\Vert V_{n}Q_{p_{n}}^{p_{n}-1}\right\Vert _{\frac{9}{2}}\le\left(\sup_{n}V_{n}\right)\sup_{n}\|Q_{p_{n}}^{p_{n}-1}\|_{\frac{9}{2}}<\wq.
\]
Therefore there exists a constant $M_{3}>0$ such that we have
\[
\sup_{n}\left\Vert (1-\la_{p_{n}})Q_{p_{n}}+V_{n}Q_{p_{n}}^{p_{n}-1}\right\Vert _{\frac{9}{2}}\le M_{3}<\wq.
\]
Thus by elliptic regularity theory, equation (\ref{eq: 2.star-2.1})
gives that $Q_{p_{n}}\in W^{2,9/2}(\R^{3})$ and
\[
\|Q_{p_{n}}\|_{W^{2,9/2}(\R^{3})}\le CM_{1}+CM_{3}<\wq,
\]
for some constant $C>0$ independent of $n$. By Sobolev embedding
theorem, we have that $W^{2,9/2}(\R^{3})\subset L^{\wq}(\R^{3})$.
Therefore, by setting $M=CM_{1}+CM_{3}$, we complete the proof of
Lemma \ref{lem: uniform bound of Q-p-n}. \end{proof}

Next we prove that $V_{n}$ decays to zero at infinite in a uniform
way.

\begin{lemma}\label{lem: uniform decay of V-n} We have that
\begin{eqnarray*}
V_{n}(x)\le C_{0}|x|^{-3/4} &  & \text{for }x\in\R^{3}.
\end{eqnarray*}
for all $n\in\N$, where $C_{0}>0$ is a constant independent of $n$.
\end{lemma}
\begin{proof}
Since $p_{n}>2$, it is easy to deduce from Lemma \ref{lem: uniform bound of Q-p-n}
that $Q_{p_{n}}^{p_{n}}\in L^{1}(\R^{3})\cap L^{\wq}(\R^{3})$ and
$\sup_{n}\big(\|Q_{p_{n}}^{p_{n}}\|_{1}+\|Q_{p_{n}}^{p_{n}}\|_{\wq}\big)<\wq$.
By the Riesz potential theory (see e.g. \cite{Stein-Book}), we obtain
that $V_{n}\in L^{4}(\R^{3})$ and
\[
\|V_{n}\|_{4}\le C\|Q_{p_{n}}^{p_{n}}\|_{\frac{12}{11}}\le C_{1}<\wq
\]
for all $n\in\N$, where $C_{1}>0$ is a constant independent of $n$.
Note that $V_{n}$ is symmetric decreasing. We obtain that
\[
C_{1}\ge\int_{B_{r}(0)}V_{n}^{4}\D x\ge V_{n}^{4}(r)\frac{4\pi}{3}r^{3}
\]
for all $r>0$ and for all $n\in\N$. By setting $C_{0}=3C_{1}/(4\pi)$,
we complete the proof of Lemma \ref{lem: uniform decay of V-n}.
\end{proof}
Note that $\|Q_{p_{n}}\|_{2}=1$ holds for all $n$ and that $Q_{p_{n}}(|x|)$
is decreasing with respect to $|x|$, we derive as above that
\begin{eqnarray}
Q_{p_{n}}(x)\le C|x|^{-3/2} &  & \text{for }x\in\R^{3}\label{eq: decay esti. of Q-p-n-1}
\end{eqnarray}
 for all $n\in\N$, where $C>0$ is a constant independent of $n$.
The estimate (\ref{eq: decay esti. of Q-p-n-1}) can be improved as
follows.

\begin{lemma}\label{lem: decay esti. of Q-p-n-2} For any $\ga>1$,
there exists a constant $C=C(\ga)>0$ such that for all $n\in\N$,
we have that
\begin{eqnarray*}
Q_{p_{n}}(x)\le C|x|^{-\ga} &  & \text{for }|x|\ge1.
\end{eqnarray*}
\end{lemma}
\begin{proof}
Note that equation (\ref{eq: 2.star-2}) gives that
\begin{equation}
Q_{p_{n}}(x)=\frac{1}{-\De+\la_{p_{n}}}V_{n}Q_{p_{n}}^{p_{n}-1}=G_{n}\ast\left(V_{n}Q_{p_{n}}^{p_{n}-1}\right),\label{eq: integral form}
\end{equation}
where $G_{n}(x)=\exp\left(-\sqrt{\la_{p_{n}}}|x|\right)/\left(4\pi|x|\right)$
is the integral kernel of the operator $\frac{1}{-\De+\la_{p_{n}}}$.
Since $\inf_{n}\la_{p_{n}}>0$ by Lemma \ref{lem: Lagrange multiplier is bounded },
there exists $\de>0$ such that
\begin{eqnarray}
G_{n}(x)\le\frac{e^{-\de|x|}}{4\pi|x|} &  & \text{for all }x\in\R^{3}\text{ and all }n\in\N.\label{esti: uniform esti of G-n}
\end{eqnarray}
Now we estimate $Q_{p_{n}}$ by virtue of (\ref{eq: integral form}).
Fix $|x|\ge1$. Then by (\ref{esti: uniform bound of V-n}) and Lemma
\ref{lem: uniform bound of Q-p-n}, we obtain that
\begin{equation}
\int_{B_{|x|/2}(0)}G_{n}(x-y)V_{n}(y)Q_{p_{n}}^{p_{n}-1}(y)\D y\le C_{1}G_{n}(|x|/2)|x|^{3}.\label{eq: esti. of part 1}
\end{equation}
Here we used the fact that $G_{n}$ is monotone decreasing with respect
to $|x|$. On the other hand, by (\ref{eq: decay esti. of Q-p-n-1}),
Lemma \ref{lem: uniform bound of Q-p-n}, Lemma \ref{lem: uniform decay of V-n}
and the fact that $p_{n}>2$, we deduce that
\begin{equation}
\int_{\R^{3}\backslash B_{|x|/2}(0)}G_{n}(x-y)V_{n}(y)Q_{p_{n}}^{p_{n}-1}(y)\D y\le C|x|^{-9/4}\int_{\R^{3}}G_{n}\D y=C|x|^{-9/4}.\label{eq: esti. of part 2}
\end{equation}
Combining (\ref{eq: esti. of part 1}) and (\ref{eq: esti. of part 2})
and noticing (\ref{esti: uniform esti of G-n}), we conclude that
\begin{eqnarray}
Q_{p_{n}}(x)\le C|x|^{-9/4} &  & \text{for all }|x|\ge1,\label{esti: uniform decay of Q-p-n-3}
\end{eqnarray}
for all $n\in\N$, where $C>0$ is independent of $n$. Note that
(\ref{esti: uniform decay of Q-p-n-3}) is an improvement of (\ref{eq: decay esti. of Q-p-n-1}).

Finally, for any given constant $\ga>9/4$, we can substituting (\ref{esti: uniform decay of Q-p-n-3})
into to (\ref{eq: esti. of part 2}) and iterating finitely many times
to deduce that $Q_{p_{n}}(x)\le C|x|^{-\ga}$ uniformly for all $n\in\N$.
The proof of Lemma \ref{lem: decay esti. of Q-p-n-2} is complete.
\end{proof}
Now we can prove Proposition \ref{prop:  Uniform estimates}.

\begin{proof}[Proof of Proposition \ref{prop:  Uniform estimates}]
We need to find the function $F$. This follows easily from Lemma
\ref{lem: uniform bound of Q-p-n} and Lemma \ref{lem: decay esti. of Q-p-n-2}.
Indeed, set
\[
F(x)=\begin{cases}
\sup_{n}\|Q_{p_{n}}\|_{\wq} & \text{for }|x|\le1,\\
C|x|^{-4} & \text{for }|x|>1,
\end{cases}
\]
where $C=C(4)>0$ is given as in Lemma \ref{lem: decay esti. of Q-p-n-2}.
Then $Q_{p_{n}}(x)\le F(x)$ holds for all $x\in\R^{3}$ and for all
$n\in\N$. Furthermore, it is easy to see that $F\in L^{1}(\R^{3})\cap L^{\wq}(\R^{3})$.
The proof of Proposition \ref{prop:  Uniform estimates} is complete.
\end{proof}

\subsection{Proof of Theorem \ref{thm: compactness result}. }

Now we are ready to prove Theorem \ref{thm: compactness result}.

Let $p_{n}$, $Q_{p_{n}}$ be defined as in Theorem \ref{thm: compactness result}.
We assume that $N=1$ so that $\|Q_{p_{n}}\|_{2}=1$ for all $n\in\N$.
Let $\la_{p_{n}}$ be defined as in (\ref{eq: m-p to lambda-p}) such
that $Q_{p_{n}}$ satisfies equation (\ref{eq: 2.star-2}). Since
$\la_{p_{n}}$ is bounded uniformly for all $n\in\N$,  (\ref{eq: Pohozaev identity})
implies that $\{Q_{p_{n}}\}$ is a bounded sequence in $H^{1}(\R^{3})$
since $\|Q_{p_{n}}\|_{2}=1$ for all $n\in\N$. Therefore we can assume,
after possibly passing to a subsequence, that $Q_{p_{n}}$ converges
weakly to a nonnegative radial function $Q_{\wq}\in H^{1}(\R^{3})$,
that is,
\begin{eqnarray}
Q_{p_{n}}\wto Q_{\wq} &  & \text{in }H^{1}(\R^{3}).\label{eq: weak convergence-1}
\end{eqnarray}
Moreover, by the compact embedding $H_{\rad}^{1}(\R^{3})\subset\subset L^{q}(\R^{3})$
for any $2<q<6$ (see Strauss \cite{Strauss-1977}), we can assume
that
\begin{eqnarray}
Q_{p_{n}}\to Q_{\wq} &  & \text{in }L^{q}(\R^{3})\label{eq: strong convergence-1}
\end{eqnarray}
for any $2<q<6$, and
\begin{eqnarray}
Q_{p_{n}}\to Q_{\wq} &  & \text{a.e. in }\R^{3}.\label{eq: pointwise estimates}
\end{eqnarray}
Furthermore, by Lemma \ref{lem: Lagrange multiplier is bounded },
we can extract a subsequence of $\{p_{n}\}$, still denote by $\{p_{n}\}$,
such that
\begin{equation}
\lim_{n\to\wq}\la_{p_{n}}=\mu\label{eq: limit of LM}
\end{equation}
for some $0<\mu<\wq$.

We claim that $\mu$ is independent of the choice of the subsequence
$\{p_{n}\}$. Indeed, by Proposition \ref{prop:  Uniform estimates},
we easily deduce that
\begin{equation}
\lim_{n\to\wq}\int_{\R^{3}}\left(|x|^{-1}\ast Q_{p_{n}}^{p_{n}}\right)Q_{p_{n}}^{p_{n}-1}\var\D y=\int_{\R^{3}}\left(|x|^{-1}\ast Q_{\wq}^{2}\right)Q_{\wq}\var\D y\label{eq: 2.star-3}
\end{equation}
for all $\var\in C_{0}^{\wq}(\R^{3})$, and that
\begin{equation}
\lim_{n\to\wq}\int_{\R^{3}}\left(|x|^{-1}\ast Q_{p_{n}}^{p_{n}}\right)Q_{p_{n}}^{p_{n}}\D y=\int_{\R^{3}}\left(|x|^{-1}\ast Q_{\wq}^{2}\right)Q_{\wq}^{2}\D y.\label{eq: 2.star-4}
\end{equation}
Then by passing to limit in equation (\ref{eq: 2.star-2}), we derive
from (\ref{eq: weak convergence-1}) and (\ref{eq: 2.star-3}) that
$Q_{\wq}$ is a solution to equation
\begin{eqnarray}
-\De Q_{\wq}-\left(|x|^{-1}\ast Q_{\wq}^{2}\right)Q_{\wq}=-\mu Q_{\wq} &  & \text{in }\R^{3}.\label{eq: equ of Q-infty}
\end{eqnarray}
We show that $\|Q_{\wq}\|_{2}=1$. By multiplying $Q_{\wq}$ on each
side of above equation and combining (\ref{eq: weak convergence-1}),
(\ref{eq: 2.star-4}), we deduce that
\begin{equation}
\begin{aligned}-\mu\int_{\R^{3}}|Q_{\wq}|^{2}\D y & =\int_{R^{3}}|\na Q_{\wq}|^{2}\D x-\int_{\R^{3}}\left(|x|^{-1}\ast Q_{\wq}^{2}\right)Q_{\wq}^{2}\D y\\
 & \le\lim_{n\to\wq}\left(\int_{R^{3}}|\na Q_{p_{n}}|^{2}\D x-\int_{\R^{3}}\left(|x|^{-1}\ast Q_{p_{n}}^{p_{n}}\right)Q_{p_{n}}^{p_{n}}\D y\right)\\
 & =-\lim_{n\to\wq}\la_{p_{n}}\\
 & =-\mu.
\end{aligned}
\label{eq: 2.star-5}
\end{equation}
Then $\|Q_{\wq}\|_{2}\ge1$ follows from (\ref{eq: 2.star-5}) since
$\mu>0$. On the other hand, since $Q_{p_{n}}\wto Q_{\wq}$ in $L^{2}(\R^{3})$,
we have that $\|Q_{\wq}\|_{2}\le1$ holds. Therefore, we conclude
that $\|Q_{\wq}\|_{2}=1$ holds. Then by the uniqueness result of
Lieb \cite{Lieb-1976}, we find that $Q_{\wq}=Q_{2}$ is the unique
positive radial solution to equation (\ref{eq: equ of Q-infty}) and
$\mu$ is determined uniquely by $Q_{\wq}$ with $\|Q_{\wq}\|_{2}=1$.
This proves the claim.

Furthermore, $\|Q_{\wq}\|_{2}=1$ implies that $Q_{p_{n}}\to Q_{\wq}$
in $L^{2}(\R^{3})$ and that the inequality in (\ref{eq: 2.star-5})
is in fact an equality. Hence we obtain that $\|\na Q_{p_{n}}\|_{2}\to\|\na Q_{\wq}\|_{2}$
as $n\to\wq$, from which we deduce that $Q_{p_{n}}\to Q_{\wq}$ in
$H^{1}(\R^{3})$ as $n\to\wq$ in view of (\ref{eq: weak convergence-1}).

Finally, to complete the proof of Theorem \ref{thm: compactness result},
we note that we have convergence along every subsequence due to the
uniqueness of the limit point $Q_{\wq}\in{\cal A}_{1}$.

\section{Proofs of main results\label{sec:Proof-of-main results}}

In this section, we prove our main results Theorem \ref{thm: Main theorem 1-Uniqueness}
and Theorem \ref{thm: Main result 2-Nondegeneracy}. First we derive
a local uniqueness result. Since our parameter $p$ varies between
$2$ and $6$, the critical Sobolev exponent for $H^{1}(\R^{3})$,
it is convenient to consider our problems in the function space
\[
\X=L_{\rad}^{2}(\R^{3})\cap L_{\rad}^{6}(\R^{3}),
\]
equipped with norm
\[
\|u\|_{\X}=\|u\|_{2}+\|u\|_{6}.
\]
Note that $H^{1}(\R^{3})$ is continuously embedded into $\X$ by
Sobolev embedding theorem.

\begin{proposition}\label{prop: Implicit theorem to local uniqueness}
Let $Q_{2}\in{\cal A}_{1}$ be the unique positive radial ground state
to equation (\ref{eq: Hartree equ.}) with $\la=\la_{2}$ given by
(\ref{eq: m-p to lambda-p}) with $p=2$. Then there exists a small
number $\de>0$ and a map $\big(\tilde{Q},\tilde{\la}\big)\in C^{1}(I;X\times\R_{+})$
defined on the interval $I=[2,2+\de)$ such that the following holds,
where we denote $\big(\tilde{Q}_{p},\tilde{\la}_{p}\big)=\big(\tilde{Q}(p),\tilde{\la}(p)\big)$
in the sequel.

(1) $\big(\tilde{Q}_{p},\tilde{\la}_{p}\big)$ is a solution to equation
(\ref{eq: Choquard equation}) with $\la=\la_{p}$ for all $p\in I$.

(2) There exists $\ep>0$ such that $\big(\tilde{Q}_{p},\tilde{\la}_{p}\big)$
is the unique solution of equation (\ref{eq: Choquard equation})
with $\la=\tilde{\la}_{p}$ for $p\in I$ in the neighborhood
\[
{\cal N}_{\ep}=\left\{ (u,\mu)\in\X\times\R_{+}:\|u-Q_{2}\|_{\X}+|\mu-2|\le\ep\right\} .
\]
In particular, we have that $\big(\tilde{Q}_{p},\tilde{\la}_{p}\big)=\left(Q_{2},\la_{2}\right)$
holds.

(3) For all $p\in I$, we have
\[
\|\tilde{Q}_{p}\|_{2}=\|Q_{2}\|_{2}=1.
\]
\end{proposition}

The proof of Proposition \ref{prop: Implicit theorem to local uniqueness}
will be given later. With the help of Proposition \ref{prop: Implicit theorem to local uniqueness},
we are able to prove Theorem \ref{thm: Main theorem 1-Uniqueness}
now.

\begin{proof}[Proof of  Theorem \ref{thm: Main theorem 1-Uniqueness}]
We assume that $N=1$. We argue by contradiction. Suppose, on the
contrary, that there exist a sequence $\{p_{n}\}\subset(2,7/3)$ with
$p_{n}\to2$ as $n\to\wq$, such that for every $p_{n}$, there exist
at least two distinct positive radial ground states $Q_{p_{n}}\in{\cal A}_{1}$
and $\bar{Q}_{p_{n}}\in{\cal A}_{1}$ for problem (\ref{eq: Choquard equation})
with $p=p_{n}$. Let $I=[2,2+\de)$ be the interval defined as in
Proposition \ref{prop: Implicit theorem to local uniqueness}. With
no loss of generality, we can assume that $\{p_{n}\}\subset I$. Let
$\la_{p_{n}}$ be defined as in (\ref{eq: m-p to lambda-p}) with
$p=p_{n}$ for all $n\in\N$. Then both $(Q_{p_{n}},\la_{p_{n}})$
and $(\bar{Q}_{p_{n}},\la_{p_{n}})$ solves equation (\ref{eq: Choquard equation})
with $p=p_{n}$ and $\la=\la_{p_{n}}$.

Applying Theorem \ref{thm: compactness result} to both sequences
$\{Q_{p_{n}}\}$ and $\{\bar{Q}_{p_{n}}\}$, we deduce that both $Q_{p_{n}}$
and $\bar{Q}_{p_{n}}$ converge strongly in $H^{1}(\R^{3})$ to the
unique positive radial ground state $Q_{2}\in{\cal A}_{1}$ of problem
(\ref{eq: Hartree equ.}). In particular, by Sobolev embedding theorem,
$Q_{p_{n}}\to Q_{2}$ and $\bar{Q}_{p_{n}}\to Q_{2}$ in $\X$ hold.
Moreover, (\ref{eq: limit of LM}) implies $\la_{p_{n}}\to\la_{2}$
as well.

Now we apply Proposition \ref{prop: Implicit theorem to local uniqueness}
to deduce a contradiction as follows. Let $\ep>0$ and the neighborhood
${\cal N}_{\ep}$ be given as in Proposition \ref{prop: Implicit theorem to local uniqueness}.
Let $(\tilde{Q}_{p_{n}},\tilde{\la}_{p_{n}})$ be defined as in Proposition
\ref{prop: Implicit theorem to local uniqueness} such that $(\tilde{Q}_{p_{n}},\tilde{\la}_{p_{n}})$
is the unique solution to equation (\ref{eq: Choquard equation})
with $p=p_{n}$ and $\la=\tilde{\la}_{p_{n}}$ in the neighborhood
${\cal N}_{\ep}$ for all $n\in\N$. Recall that both $(Q_{p_{n}},\la_{p_{n}})$
and $(\bar{Q}_{p_{n}},\la_{p_{n}})$ are positive radial solutions
to equation (\ref{eq: Choquard equation}) with $p=p_{n}$ and $\la=\la_{p_{n}}$
for all $n\in\N$. Recall also that $\la_{p_{n}}$ is independent
of the choice of $Q_{p_{n}}$ or $\bar{Q}_{p_{n}}$ for all $n\in\N$
in view of (\ref{eq: m-p to lambda-p}). Since both $Q_{p_{n}}\to Q_{2}$
and $\bar{Q}_{p_{n}}\to Q_{2}$ in $\X$ and $\la_{p_{n}}\to\la_{2}$
hold as $n\to\wq$, we find that $(Q_{p_{n}},\la_{p_{n}})\in{\cal N}_{\ep}$
and $(\bar{Q}_{p_{n}},\la_{p_{n}})\in{\cal N}_{\ep}$ hold for all
sufficiently large $n$. Therefore by Proposition \ref{prop: Implicit theorem to local uniqueness},
we deduce for all sufficiently large $n$ that
\begin{eqnarray*}
Q_{p_{n}}=\bar{Q}_{p_{n}}\equiv\tilde{Q}_{p_{n}} & \text{and} & \la_{p_{n}}=\tilde{\la}_{p_{n}}.
\end{eqnarray*}
We reach a contradiction since we assumed that $Q_{p_{n}}\neq\bar{Q}_{p_{n}}$
for all $n\in\N$. The proof of Theorem \ref{thm: Main theorem 1-Uniqueness}
is complete now. \end{proof}

We remark that from above proof of Theorem \ref{thm: Main theorem 1-Uniqueness},
we find that $(Q_{p},\la_{p})=\big(\tilde{Q}_{p},\tilde{\la}_{p}\big)$
for $2<p<2+\de$, where $Q_{p}\in{\cal A}_{1}$ is the unique positive
radial ground state for problem (\ref{eq: Choquard equation}) and
$\la_{p}$ is the corresponding Lagrange multiplier. Thus $p\mapsto(Q_{p},\la_{p})$
is $C^{1}$ for $2<p<2+\de$.

It remains to prove Proposition \ref{prop: Implicit theorem to local uniqueness}.
We use an implicit function argument. We follow the line of Frank
and Lenzmann \cite[Proposition 5.2]{Frank-Lenzmann-2013}.

\begin{proof}[Proof of Proposition \ref{prop: Implicit theorem to local uniqueness}]
Observe that $u\in\X$ is a solution to equation (\ref{eq: Choquard equation})
if and only if
\begin{eqnarray*}
u-\frac{1}{-\De+\la}\left(|x|^{-1}\ast|u|^{p}\right)|u|^{p-2}u=0 &  & \text{in }\R^{3}.
\end{eqnarray*}
Here $\frac{1}{-\De+\la}$ denotes the bounded inverse operator of
$-\De+\la$ on $L^{2}(\R^{3})$ for $\la>0$. For $0<\de<1/3$ sufficiently
small, we define the map $F:\X\times\R_{+}\times[2,2+\de)\to\X\times\R$
by
\[
F(u,\la,p)=\left(\begin{array}{c}
u-{\displaystyle \frac{1}{-\De+\la}}\left(|x|^{-1}\ast|u|^{p}\right)|u|^{p-2}u\\
\|u\|_{2}^{2}-\|Q_{2}\|_{2}^{2}
\end{array}\right).
\]
By Lemma \ref{lem: Regularity of F}, $F$ is continuously Fr\'echet
differentiable, and $\pa_{(u,\la)}F:\X\times\R\to X\times R$ is given
by
\[
\pa_{(u,\la)}F=\left(\begin{array}{cc}
\Id+K_{p}, & W_{p}\\
2\langle u,\cdot\rangle, & 0
\end{array}\right),
\]
where $K_{p}$ is given by
\[
K_{p}\xi=-\frac{1}{-\De+\la}\Big((p-1)\left(|x|^{-1}\ast|u|^{p}\right)|u|^{p-2}\xi+p\left(|x|^{-1}\ast\left(|u|^{p-2}u\xi\right)\right)|u|^{p-2}u\Big),
\]
$W_{p}$ is given by
\[
W_{p}=\frac{1}{\left(-\De+\la\right)^{2}}\left(|x|^{-1}\ast|u|^{p}\right)|u|^{p-2}u,
\]
 and $2\langle u,\cdot\rangle:\X\to\R$ denotes the mapping $2\langle u,v\rangle=2\int_{\R^{3}}uv\D x$.

Consider the derivative $\pa_{(u,\la)}F$ at the point $(u,\la,p)=(Q_{2},\la_{2},2)$.
For simplicity, we write $T=\pa_{(u,\la)}F|_{(u,\la,p)=(Q_{2},\la_{2},2)}$.
We claim that the inverse of $T:\X\times\R\to\X\times\R$ exists.
That is, we have to show that for any $(f,\al)\in\X\times\R$ given,
there exists a unique $(g,\be)\in\X\times\R$ such that the following
system is satisfied:
\begin{alignat}{1}
 & (\Id+K_{2})g+W_{2}\beta=f,\label{eq: the first one}\\
 & 2\langle Q_{2},g\rangle=\al.\label{eq: The second one}
\end{alignat}
Note that $K_{2}$ is given by
\[
K_{2}\xi=-\frac{1}{-\De+\la_{2}}\Big(\left(|x|^{-1}\ast Q_{2}^{2}\right)\xi+2\left(|x|^{-1}\ast\left(Q_{2}\xi\right)\right)Q_{2}\Big),
\]
and $W_{2}$ is given by
\[
W_{2}=\frac{1}{\left(-\De+\la_{2}\right)^{2}}\left(|x|^{-1}\ast Q_{2}^{2}\right)Q_{2}.
\]
It is straightforward to verify that $K_{2}$ satisfies the identity
\begin{equation}
\Id+K_{2}=(-\De+\la_{2})^{-1}\L_{+,2},\label{eq: inverse of Id plus K}
\end{equation}
where $\L_{+,2}$ is defined as in (\ref{eq: linearized operator})
with $p=2$, and $W_{2}$ satisfies the identity
\begin{equation}
(-\De+\la_{2})W_{2}=Q_{2}.\label{eq: 4.star-1}
\end{equation}

We claim that $\Id+K_{2}$ has a bounded inverse on $L_{\rad}^{2}(\R^{3})$.
Otherwise, $-1$ belongs to the spectrum of $K_{2}$. Since $K_{2}$
is a compact operator on $L_{\rad}^{2}(\R^{3})$, we know that $-1$
is an eigenvalue of $K_{2}$. Thus there exists $v\in L_{\rad}^{2}(\R^{3})$,
$v\ne0$, such that $(\Id+K_{2})v=0$. But then (\ref{eq: inverse of Id plus K})
gives that $\L_{+,2}v=0$. However, by the nondegeneracy result of
Lenzmann \cite{Lenzmann-2009} (that is, Theorem \ref{thm: Main result 2-Nondegeneracy}
with $p=2$), we have that $v\equiv0$. We obtain a contradiction.
This proves the claim. Moreover, since $K_{2}:\X\to\X$ holds (see
the proof of Lemma \ref{lem: Regularity of F} for details), we deduce
that $\left(\Id+K_{2}\right)^{-1}$ exists on the space $\X$ as well.
Hence we can solve equation (\ref{eq: the first one}) for $g$ uniquely
by
\[
g=\left(\Id+K_{2}\right)^{-1}\left(f-W_{2}\beta\right).
\]
Combining this equation together with (\ref{eq: The second one})
yields
\[
2\langle Q_{2},\left(\Id+K_{2}\right)^{-1}W_{2}\rangle\beta=2\langle Q_{2},\left(\Id+K_{2}\right)^{-1}f\rangle-\al.
\]
Thus, to solve $\beta$ uniquely, it is equivalent to show that $2\langle Q_{2},\left(\Id+K_{2}\right)^{-1}W_{2}\rangle\ne0$.
To see this, we use the fact that
\begin{equation}
\L_{+,2}R=-2\la_{2}Q_{2},\label{eq: 3.star-1}
\end{equation}
where $R=2Q_{2}+x\cdot\na Q_{2}$ (see (4-28) of Lenzmann \cite{Lenzmann-2009}).
Then using the identities (\ref{eq: inverse of Id plus K}) (\ref{eq: 4.star-1})
and (\ref{eq: 3.star-1}) gives us that
\[
2\langle Q_{2},\left(\Id+K_{2}\right)^{-1}W_{2}\rangle=-\frac{1}{2\la_{2}}\int_{\R^{3}}|Q_{2}|^{2}\D x\ne0.
\]
This proves the claim that $T$ has an inverse mapping. Finally, applying
the implicit function theorem to the map $F$ at $(Q_{2},\la_{2},2)$
as that of Frank and Lenzmann \cite[Proposition 5.2]{Frank-Lenzmann-2013},
we derive the assertions (1)-(3) provided that $\de>0$ is sufficiently
small. The proof of Proposition \ref{prop: Implicit theorem to local uniqueness}
is complete. \end{proof}

Next we prove Theorem \ref{thm: Main result 2-Nondegeneracy}. We
follow the argument of Lenzmann \cite[Theorem 3]{Lenzmann-2009}.

\begin{proof}[Proof of  Theorem \ref{thm: Main result 2-Nondegeneracy}]
Assume that $N=1$. Let $(Q_{2},\la_{2})$ be the unique positive
radial ground state to equation (\ref{eq: Choquard equation}) with
$\la=\la_{2}$ and consider the linear operator $\L_{+,2}$ associated
to $Q_{2}$ defined as in (\ref{eq: linearized operator}) with $p=2$.
Then it was given by Lenzmann \cite[Theorem 4]{Lenzmann-2009} that
\[
\text{{\rm Ker}}\L_{+,2}=\text{{\rm span}}\left\{ \pa_{x_{1}}Q_{2},\pa_{x_{2}}Q_{2},\pa_{x_{3}}Q_{2}\right\} .
\]
On the other hand, let $\de>0$ be defined as in Theorem \ref{thm: Main theorem 1-Uniqueness}
and consider $2<p<2+\de$. Let $(Q_{p},\la_{p})$ be the unique positive
ground state to equation (\ref{eq: Choquard equation}) with $\la=\la_{p}$.
Consider the linear operator $\L_{+,p}$ associated to $Q_{p}$ defined
as in (\ref{eq: linearized operator}). By differentiation, we deduce
that
\begin{equation}
\text{{\rm span}}\left\{ \pa_{x_{1}}Q_{p},\pa_{x_{2}}Q_{p},\pa_{x_{3}}Q_{p}\right\} \subseteq\text{{\rm Ker}}\L_{+,p}.\label{eq: 4.star-2}
\end{equation}
Our aim is to show that above equality is attained. The idea is to
show that the dimension of $\text{{\rm Ker}}\L_{+,p}$ is at most
three. We use the following standard perturbation argument.

As pointed out by Lenzmann \cite{Lenzmann-2009}, $0$ is an isolated
eigenvalue of the spectrum of $\L_{+,2}$. Consider the integral of
the resolvent of $\L_{+,2}$
\[
P_{0}=\frac{1}{2\pi i}\oint_{\pa D_{r}}(\L_{+,2}-z)^{-1}\D z,
\]
the so called Riesz projection $P_{0}$ of $\L_{+,2}$ on its kernel
$\text{{\rm Ker}}\L_{+,2}$, where $D_{r}=\{z\in\C:|z|<r\}$. Here
we choose $r$ sufficiently mall such that $0$ is the unique eigenvalue
of $\L_{+,2}$ on the closed ball $\bar{D}_{r}$. We claim that the
projection
\[
P_{0,p}=\frac{1}{2\pi i}\oint_{\pa D_{r}}(\L_{+,p}-z)^{-1}\D z
\]
exists for $2<p<2+\de^{\prime}$, where $0<\de^{\prime}\le\de$ is
a sufficiently small number, and satisfies
\begin{eqnarray}
\|P_{0,p}-P_{0}\|_{L^{2}\to L^{2}}\to0 &  & \text{as }p\to2.\label{eq: approximation of projections}
\end{eqnarray}
Indeed, we conclude by the remark that follows the proof of Theorem
\ref{thm: compactness result} and Lemma \ref{lem: Regularity of F}
that
\[
\|(\L_{+,p}-z)^{-1}\|_{L^{2}\to L^{2}}\le C\|(\L_{+,2}-z)^{-1}\|_{L^{2}\to L^{2}}
\]
 for $p>2$ and sufficiently close to $2$ and $z\in\pa D_{r}$, where
$C>0$ is a constant, and that
\begin{eqnarray*}
\|(\L_{+,p}-z)^{-1}-(\L_{+,2}-z)^{-1}\|_{L^{2}\to L^{2}}\to0 &  & \text{as }p\to2.
\end{eqnarray*}
This shows that $P_{0,p}$ exists for $2<p<2+\de^{\prime}$ and (\ref{eq: approximation of projections})
holds, provided that $0<\de^{\prime}\le\de$ is sufficiently small.
Since $\text{rank\,}P_{0}=3$ and the rank of $P_{0,p}$ remains constant
for $2<p<2+\de^{\prime}$, we deduce by (\ref{eq: approximation of projections})
that $P_{0,p}$ has at most 3 eigenvalues (counted with their multiplicity)
on $\bar{D}_{r}$ for $2<p<2+\de^{\prime}$. In particular, we obtain
that
\[
\text{{\rm dim}}\text{Ker}\L_{+,p}\le3
\]
for $2<p<2+\de^{\prime}$. Therefore the equality in (\ref{eq: 4.star-2})
must hold for $2<p<2+\de^{\prime}$. This finishes the proof of Theorem
\ref{thm: Main result 2-Nondegeneracy}.\end{proof}

\appendix

\section{Proof of Theorem \ref{thm: properties of GS}\label{sec: Existence}}

In this section, we give a short proof for the existence part of Theorem
\ref{thm: properties of GS} for the sake of completeness. A complete
proof of Theorem \ref{thm: properties of GS} can be found in Moroz
and Van Schaftingen \cite{Moroz-Van Schaftingen-2013}. We start the
proof with the following observation.

\begin{lemma}\label{lem: finite infimum} Assume that $5/3<p<7/3$.
Then $-\wq<m(N,p)<0$ for any $N>0$.\end{lemma}
\begin{proof}
It is a consequence of Lemma \ref{lem: equivalent minimum} and Lemma
\ref{lem: lower bound of m-p}.
\end{proof}
By Lemma \ref{lem: N is 1}, we have the following conclusion.

\begin{lemma} \label{lem: infimum is decreasing}The infimum $m(N,p)$
is strictly decreasing with respect to $N$. \end{lemma}

To prove the existence of minimizers of problem (\ref{prob: minimizing}),
we apply the rearrangement technique. For any given function $u\in H^{1}(\R^{3})$,
we denote by $u^{*}$ the symmetric-decreasing rearrangement of $u$.
The following properties hold for all $u\in H^{1}(\R^{3})$
\begin{eqnarray}
 &  & \|u^{*}\|_{q}=\|u\|_{q}\quad\forall\:2\le q\le6;\label{eq: unchange quantity}\\
 &  & K(u^{*})\le K(u);\label{eq: Kinetic energy decreasing}\\
 &  & D_{p}(u^{*})\ge D_{p}(u).\label{eq: Riesz inequ.}
\end{eqnarray}
Equality of (\ref{eq: Riesz inequ.}) is attained if and only if $u(x)=u^{*}(x-x_{0})$
for some $x_{0}\in\R^{3}$. For the precise definition of $u^{*}$
and the proof of above properties, we refer to e.g. Lieb and Loss
\cite{Lieb-Loss-book}. Now we are ready to prove Theorem \ref{thm: properties of GS}.

\begin{proof}[Proof of Theorem \ref{thm: properties of GS}] By Lemma
\ref{lem: N is 1}, it suffices to prove Theorem \ref{thm: properties of GS}
in the case $N=1$. Recall that we denote $m(p)=m(1,p)$. That is,
\begin{equation}
m(p)=\inf\left\{ E_{p}(u):u\in{\cal A}_{1}\right\} .\label{prob: equivalent problem 2}
\end{equation}

Our first aim is to show that $m(p)$ is attained. Let $\{u_{n}\}\subset{\cal A}_{1}$
be a minimizing sequence of problem (\ref{prob: equivalent problem 2}).
Consider the symmetric-decreasing rearrangement $u_{n}^{*}$ for all
$n\in\N$. By (\ref{eq: unchange quantity}) (\ref{eq: Kinetic energy decreasing})
and (\ref{eq: Riesz inequ.}), we deduce that the sequence $\{u_{n}^{*}\}$
is also a minimizing sequence of problem (\ref{prob: equivalent problem 2}).
On the other hand, since $0<3p-5<2$, we deduce from (\ref{eq: Interpolation inequ.})
that
\begin{eqnarray*}
E_{p}(u)\ge\frac{1}{4}\int_{\R^{3}}|\na u|^{2}\D x-C_{p} &  & \forall\, u\in{\cal A}_{1},
\end{eqnarray*}
for a constant $C_{p}>0$ depending only on $p$. Therefore we conclude
by using above estimate that $\{u_{n}^{*}\}$ is a bounded sequence
in $H^{1}(\R^{3})$. Hence there exists a function $u\in H^{1}(\R^{3})$
such that
\begin{eqnarray}
u_{n}^{*}\wto u &  & \text{in }H^{1}(\R^{3}).\label{eq: weak convergence}
\end{eqnarray}
Moreover, we have that $u\in H_{\rad}^{1}(\R^{3})$ and $u\ge0$ hold
since $u_{n}^{*}$ are nonnegative radial functions for all $n\in\N$.
Furthermore, by the compactness embedding $H_{\rad}^{1}(\R^{3})\subset L^{q}(\R^{3})$
for any $2<q<6$ (see Strauss \cite{Strauss-1977}), we can assume
by passing to a subsequence that
\begin{eqnarray}
u_{n}^{*}\to u &  & \text{ in }L^{6p/5}(\mathbb{R}^{3}).\label{eq: strong convergence in Lp}
\end{eqnarray}
Then it follows from (\ref{eq: weak convergence}) that
\[
\int_{\mathbb{R}^{3}}|\nabla u_{n}^{*}|^{2}\D x=\int_{\mathbb{R}^{3}}|\nabla v_{n}|^{2}\D x+\int_{\mathbb{R}^{3}}|\nabla u|^{2}\D x+o(1),
\]
where $v_{n}=u_{n}^{*}-u$, and it follows from (\ref{eq: strong convergence in Lp})
that
\[
\int_{\mathbb{R}^{3}}\left(I_{2}\ast u_{n}^{*p}\right)u_{n}^{*p}\D x=\int_{\mathbb{R}^{3}}\left(I_{2}\ast u^{p}\right)u^{p}\D x+o(1)
\]
as $n\to\wq$. Therefore, combining above two equalities gives us
that
\begin{equation}
E_{p}(u_{n}^{*})=K(v_{n})+E_{p}(u)+o(1)\label{eq: energy decomposition}
\end{equation}
as $k\to\wq$. Since $K(v_{n})\ge0$, we obtain that $E_{p}(u)\le m(p)<0$.
In particular, we obtain that $u\not\equiv0$.

To show that $E_{p}(u)=m(p)$, it suffices to show that $u\in\mathcal{A}_{1}$.
Suppose that $u\notin\mathcal{A}_{1}$ holds. Since $u_{n}^{*}\wto u$
in $L^{2}(\R^{3})$, we have $\|u\|_{2}\le\liminf_{n}\|u_{n}^{*}\|_{2}=1$.
Hence there exists $\alpha\in(0,1)$ such that $\|u\|_{2}=\alpha$,
that is, $u\in{\cal A}_{\al}$. Then $m(\al,p)\le E_{p}(u)$ holds.
Note that $m(\al,p)>m(1,p)=m(p)$ by Lemma \ref{lem: infimum is decreasing}.
We obtain that
\[
m(p)<m(\al,p)\le E_{p}(u)\le m(p),
\]
which is impossible! Hence $u\in\mathcal{A}_{1}$ holds. Then we obtain
$E_{p}(u)=m(p)$. This shows that $u$ is a minimizer of problem (\ref{prob: equivalent problem 2}).

Suppose now $Q\in{\cal A}_{1}$ is an arbitrary minimizer of problem
(\ref{prob: equivalent problem 2}). Then we have $Q^{*}\in{\cal A}_{1}$
by (\ref{eq: unchange quantity}), which implies that $E_{p}(Q^{*})\ge E_{p}(Q)$.
On the other hand, by (\ref{eq: Kinetic energy decreasing}) and (\ref{eq: Riesz inequ.})
we derive $E_{p}(Q^{*})\le E_{p}(Q)$. Hence, we have $E_{p}(Q^{*})=E_{p}(Q)$,
from which we infer that $Q^{*}$ is also a minimizer of problem (\ref{prob: equivalent problem 2}),
and that the equalities in (\ref{eq: Kinetic energy decreasing})
(\ref{eq: Riesz inequ.}) are attained at $u=Q$. Therefore, there
exists a point $x_{0}\in\R^{3}$ such that $Q(x)=Q^{*}(x-x_{0})$.
This proves that every minimizer of problem (\ref{prob: equivalent problem 2})
is a nonnegative radial function with respect to a point $x_{0}\in\R^{3}$
and symmetric-decreasing with respect to $r=|x-x_{0}|$.

Next we prove that $Q$ is positive everywhere. It is well known that
$Q$ solves equation (\ref{eq: Choquard equation}) with a positive
Lagrange multiplier $\la>0$. Thus we obtain from equation (\ref{eq: Choquard equation})
that $Q$ satisfies
\[
Q=\frac{1}{-\De+\la}\left(|x|^{-1}\ast Q^{p}\right)Q^{p-1}.
\]
As the integral kernel of $\frac{1}{-\De+\la}$ is positive everywhere
and $Q$ is nonnegative nontrivial, we infer from above formula that
$Q$ is strictly positive in $\R^{3}$.

Now applying Proposition 4.1 of Moroz and Schaftingen \cite{Moroz-Van Schaftingen-2013},
we obtain that $Q\in W^{2,s}(\R^{3})\cap C^{\wq}(\R^{3})$ for any
$s>1$. By the maximum principle, we conclude that $Q^{\prime}(|x|)<0$
for $|x|\neq0$. The last assertion of Theorem \ref{thm: properties of GS}
is covered by Theorem 4 of Moroz and Schaftingen \cite{Moroz-Van Schaftingen-2013}.
We omit the details. The proof of Theorem \ref{thm: properties of GS}
is complete. \end{proof}

\section{Regularity of $F$\label{sec: Regularity-of F}}

Recall that in Section \ref{sec:Proof-of-main results} we denote
\[
\X=L_{\rad}^{2}(\R^{3})\cap L_{\rad}^{6}(\R^{3})
\]
 equipped with norm $\|u\|_{\X}=\|u\|_{2}+\|u\|_{6}$. Define the
map $F:\X\times\R_{+}\times[2,7/3)\to\X\times\R$ by
\[
F(u,\la,p)=\left(\begin{array}{c}
u-{\displaystyle \frac{1}{-\De+\la}}\left(|x|^{-1}\ast|u|^{p}\right)|u|^{p-2}u\\
\|u\|_{2}^{2}-c_{0}
\end{array}\right),
\]
for $u\in\X$, $\la\in\R_{+}$ and $p\in[2,7/3)$. Here $c_{0}$ is
a fixed constant. For simplicity, we write $I=[2,7/3)$ below. Since
we will make a series of estimates, it is convenient to use the standard
notation $A\apprle B$ to denote $A\le CB$ for some constant $C>0$
that only depends on some fixed quantities. We also write $A\apprle_{a,b,\ldots}B$
to underline that $C$ depends on the fixed quantities $a,b,\ldots$
etc.

\begin{lemma}\label{lem: Regularity of F} The map $F:\X\times\R_{+}\times I\to\X\times\R$
is $C^{1}$. \end{lemma}
\begin{proof}
First, we prove that $F:\X\times\R_{+}\times I\to\X\times\R$ is well
defined. Note that there holds
\begin{eqnarray}
\left\Vert {\displaystyle \frac{1}{-\De+\la}}v\right\Vert _{H^{2}(\R^{3})}\lesssim_{\la}\|v\|_{2} &  & \text{for }v\in L^{2}(\R^{3}).\label{eq: B. Sobolev ineq}
\end{eqnarray}
Since $H^{2}(\R^{3})$ is continuously embedded in $L^{2}(\R^{3})\cap L^{\wq}(\R^{3})$,
$H^{2}(\R^{3})$ is continuously embedded in $\X$ continuously as
well. Thus we only need to show that $\left(|x|^{-1}\ast|u|^{p}\right)|u|^{p-2}u\in L_{\rad}^{2}(\R^{3})$
holds for any $u\in\X$ and $p\in I$. For notational simplicity,
write
\[
g(u,p)=\left(|x|^{-1}\ast|u|^{p}\right)|u|^{p-2}u.
\]
Set $q=6(2p-1)/7$. Then we have $2\le p<q<6$ since $p\in I$. By
interpolation, it is elementary to compute that
\begin{equation}
\|u\|_{q}\le\|u\|_{\X}\label{eq: B. interpolation 1}
\end{equation}
 for $u\in\X$. Then Young's inequality gives us that
\begin{equation}
\left\Vert |x|^{-1}\ast|u|^{p}\right\Vert _{r}\apprle_{p}\|u\|_{q}^{p}\apprle_{p}\|u\|_{\X}^{p}\label{eq: A2}
\end{equation}
where $r$ is given by
\begin{equation}
\frac{1}{r}+1=\frac{1}{3}+\frac{p}{q}.\label{eq: Young's index}
\end{equation}
It is elementary to obtain that
\begin{equation}
\left\Vert u|^{p-2}u\right\Vert _{\frac{q}{p-1}}=\|u\|_{q}^{p-1}\le\|u\|_{\X}^{p-1}.\label{eq: A3}
\end{equation}
 Hence Combining (\ref{eq: A2}) and (\ref{eq: A3}) yields that
\begin{equation}
\|g(u,p)\|_{2}\apprle_{p}\|u\|_{\X}^{2p-1}.\label{eq: A.1}
\end{equation}
Thus, by (\ref{eq: A.1}) and (\ref{eq: B. Sobolev ineq}) we deduce
that
\[
\left\Vert {\displaystyle \frac{1}{-\De+\la}}g(u,p)\right\Vert _{\X}\apprle\left\Vert {\displaystyle \frac{1}{-\De+\la}}g(u,p)\right\Vert _{H^{2}}\lesssim_{\la}\|g(u,p)\|_{2}\lesssim_{\la,p}\|u\|_{\X}^{2p-1}.
\]
This proves that $F$ is well defined.

Next we turn to the Fr\'echet differentiability of $F$. It is straightforward
to verify that the second component $F_{2}$ of $F$ is continuously
Fr\'echet differentiable and its Fr\'echet derivative at $u$ is
given by $F_{2}^{\prime}(u)=2\langle u,\cdot\rangle$, where $\langle u,\cdot\rangle$
denotes the map $g\mapsto\langle u,g\rangle$. Let us now turn to
consider the Fr\'echet differentiability the first component
\[
F_{1}=u-{\displaystyle \frac{1}{-\De+\la}}g(u,p).
\]
 We claim that $F_{1}\in C^{1}$ and its partial derivatives are given
by
\begin{eqnarray*}
\frac{\pa F_{1}}{\pa u}=1-{\displaystyle \frac{1}{-\De+\la}}g_{u}(u,p), &  & \frac{\pa F_{1}}{\pa\la}=\frac{1}{\left(-\De+\la\right)^{2}}g(u,p)
\end{eqnarray*}
and
\[
\frac{\pa F_{1}}{\pa p}=-{\displaystyle \frac{1}{-\De+\la}}g_{p}(u,p),
\]
where $g_{u}(u,p)=\pa_{u}g(u,p):\X\to\X$ is given by
\[
g_{u}(u,p)f=\left(|x|^{-1}\ast\left(p|u|^{p-2}uf\right)\right)|u|^{p-2}u+\left(|x|^{-1}\ast|u|^{p}\right)(p-1)|u|^{p-2}f,
\]
and
\[
g_{p}(u,p)=\frac{\pa g(u,p)}{\pa p}=\left(|x|^{-1}\ast\left(|u|^{p}\log|u|\right)\right)|u|^{p-2}u+\left(|x|^{-1}\ast|u|^{p}\right)|u|^{p-2}u\log|u|.
\]

This claim follows in a standard way by using Sobolev inequalities,
H\"older's inequality and estimates such as (\ref{eq: B. interpolation 1})
(\ref{eq: A2}) (\ref{eq: A3}) and the regularity of functions such
as $t\mapsto t^{p-1}$ with $p\ge2$. In the following we prove the
claim of $\pa F_{1}/\pa u$. The claims of other two partial derivatives
$\pa_{\la}F_{1},\pa_{p}F_{1}$ can be proved similarly.

First we prove that $\pa F_{1}/\pa u$ exists and be given as above.
So we have to show that for any $h\in\X$,
\begin{equation}
F_{1}(u+h,\la,p)-F_{1}(u,\la,p)-\frac{\pa F_{1}}{\pa u}(u,\la,p)h=o(1)h,\label{eq: A4}
\end{equation}
where $o(1)\to0$ as $\|h\|_{\X}\to0$. By a direct calculation, we
obtain that
\[
F_{1}(u+h,\la,p)-F_{1}(u,\la,p)-\frac{\pa F_{1}}{\pa u}(u,\la,p)h=-{\displaystyle \frac{1}{-\De+\la}}\sum_{i=1}^{3}M_{i},
\]
where $M_{i}$, $i=1,2,3$, are given by
\begin{eqnarray*}
 &  & M_{1}=\left(|x|^{-1}\ast\left(|u+h|^{p}-|u|^{p}-p|u|^{p-2}uh\right)\right)|u+h|^{p-2}(u+h),\\
 &  & M_{2}=\left(|x|^{-1}\ast|u|^{p}\right)\left(|u+h|^{p-2}(u+h)-|u|^{p-2}u-(p-1)|u|^{p-2}h\right),\\
 &  & M_{3}=\left(|x|^{-1}\ast p|u|^{p-2}uh\right)\left(|u+h|^{p-2}(u+h)-|u|^{p-2}u\right),
\end{eqnarray*}
respectively. We always assume that $\|h\|_{\X}\le1$ since we let
$\|h\|_{\X}$ tend to zero in the end. Note that
\begin{equation}
\left||u+h|^{p}-|u|^{p}-p|u|^{p-2}uh\right|\apprle_{p}(|u|^{p-2}+|h|^{p-2})|h|^{2},\label{eq: A. estimate 1}
\end{equation}
Then by (\ref{eq: A2}) (\ref{eq: A. estimate 1}) and (\ref{eq: B. interpolation 1})
we deduce that
\[
\left\Vert |x|^{-1}\ast\left(|u+h|^{p}-|u|^{p}-p|u|^{p-2}uh\right)\right\Vert _{r}\apprle_{p,\|u\|_{\X}}\|h\|_{\X}^{2},
\]
where $r$ is defined as in (\ref{eq: Young's index}) with $q=6(2p-1)/7$.
Then combining above estimate and (\ref{eq: A3}) as before implies
that
\[
\|M_{1}\|_{2}\apprle_{p,\|u\|_{\X}}\|h\|_{\X}^{2}.
\]
Similarly, since $p\ge2$, we deduce that
\[
\|M_{2}\|_{2}\apprle_{p,\|u\|_{\X}}o(1)\|h\|_{\X},
\]
where $o(1)\to0$ as $\|h\|_{\X}\to0$, and that
\[
\|M_{3}\|_{2}\apprle_{p,\|u\|_{\X}}\|h\|_{\X}^{2}.
\]
Therefore, we obtain by combining above three estimates together that

\[
\left\Vert F_{1}(u+h,\la,p)-F_{1}(u,\la,p)-\frac{\pa F_{1}}{\pa u}(u,\la,p)\right\Vert _{\X}\apprle_{p,\|u\|_{\X}}o(1)\|h\|_{\X}.
\]
This proves (\ref{eq: A4}) and thus $\pa_{u}F_{1}$ exists and be
given as claimed.

Next we prove that $\pa_{u}F_{1}$ depends continuously on $(u,\la,p)$.
Fix an arbitrary point $(u,\la,p)\in\X\times\R_{+}\times I$ and let
$\ep>0$. We have to find $\de>0$ such that
\begin{equation}
\left\Vert \left(\frac{\pa F_{1}}{\pa u}(u,\la,p)-\frac{\pa F_{1}}{\pa u}(\tilde{u},\tilde{\la},\tilde{p})\right)f\right\Vert _{\X}\le\ep\|f\|_{\X},\label{eq: A. target 1}
\end{equation}
whenever $\|u-\tilde{u}\|_{\X}+|\la-\tilde{\la}|+|p-\tilde{p}|\le\de$
holds for $(\tilde{u},\tilde{\la},\tilde{p})\in\X\times\R_{+}\times I$.

Note that
\[
\begin{aligned}\left\Vert \left(\frac{\pa F_{1}}{\pa u}(u,\la,p)-\frac{\pa F_{1}}{\pa u}(\tilde{u},\tilde{\la},\tilde{p})\right)f\right\Vert _{\X} & \le\left\Vert \frac{1}{-\De+\la}\Big(g_{u}(\tilde{u},\tilde{p})-g_{u}(u,p)\Big)f\right\Vert _{\X}\\
 & \;+\left\Vert \left(\frac{1}{-\De+\tilde{\la}}-{\displaystyle \frac{1}{-\De+\la}}\right)g_{u}(\tilde{u},\tilde{p})f\right\Vert _{\X}\\
 & =:J_{1}+J_{2}.
\end{aligned}
\]
In the following we only prove (\ref{eq: A. target 1}) for the first
term $J_{1}$. That is, there exists $\de>0$, such that whenever
$\|u-\tilde{u}\|_{\X}+|\la-\tilde{\la}|+|p-\tilde{p}|\le\de$ holds
for $(\tilde{u},\tilde{\la},\tilde{p})\in\X\times\R_{+}\times I$,
then
\begin{equation}
J_{1}\equiv\left\Vert \frac{1}{-\De+\la}\Big(g_{u}(\tilde{u},\tilde{p})-g_{u}(u,p)\Big)f\right\Vert _{\X}\le\ep\|f\|_{\X}.\label{eq: A. estimate of J-1}
\end{equation}
The estimate of $J_{2}$ can be derived in the same way as that of
Frank and Lenzmann \cite[Appendix E]{Lenzmann-2009}, where even more
general operators are considered. For example, the $C^{1}$ continuity
about the parameters $s$ and $\la$ of the operator $\left(\left(-\De\right)^{s}+\la\right)^{-1}$
is proven in Frank and Lenzmann \cite[Lemma E.1]{Lenzmann-2009}.

To prove (\ref{eq: A. estimate of J-1}), (\ref{eq: B. Sobolev ineq})
implies that it is sufficient to prove
\begin{equation}
\left\Vert \Big(g_{u}(\tilde{u},\tilde{p})-g_{u}(u,p)\Big)f\right\Vert _{2}\apprle_{\la,p}\ep\|f\|_{\X}\label{eq: A.esti. of J-1-1}
\end{equation}
whenever $\|u-\tilde{u}\|_{\X}+|p-\tilde{p}|\le\de$ holds $1>\de>0$
small enough and $\left(\tilde{u},\tilde{p}\right)\in\X\times I$.
Note that
\[
\begin{aligned}\left\Vert \left(g_{u}(\tilde{u},\tilde{p})-g_{u}(u,p)\right)f\right\Vert _{2} & \le\left\Vert \Big(g_{u}(\tilde{u},\tilde{p})-g_{u}(u,\tilde{p})\Big)f\right\Vert _{2}+\left\Vert \Big(g_{u}(u,\tilde{p})-g_{u}(u,p)\Big)f\right\Vert _{2}\end{aligned}
.
\]
Denote
\begin{eqnarray*}
L_{1}=\Big(g_{u}(\tilde{u},\tilde{p})-g_{u}(u,\tilde{p})\Big)f & \text{and} & L_{2}=\Big(g_{u}(u,\tilde{p})-g_{u}(u,p)\Big)f.
\end{eqnarray*}
We show that
\begin{equation}
\|L_{1}\|_{2}\apprle_{p,\|u\|_{\X}}\ep\|f\|_{\X}\label{eq: A.esti-J-1}
\end{equation}
and that
\begin{equation}
\|L_{2}\|_{2}\apprle_{p,\|u\|_{\X}}\ep\|f\|_{\X}\label{eq: A.esti-J-2}
\end{equation}
hold. In the sequel we estimate $L_{1}$ and $L_{2}$ one by one.

First we estimate $L_{1}$. it is easy to obtain that $L_{1}=\sum_{i=1}^{4}L_{1i},$where
\begin{eqnarray*}
 &  & L_{11}=\left(|x|^{-1}\ast|u|^{\tilde{p}}\right)(\tilde{p}-1)\left(|\tilde{u}|^{\tilde{p}-2}-|u|^{\tilde{p}-2}\right)f,\vspace{0.2cm}\\
 &  & L_{12}=\left(|x|^{-1}\ast\left(|\tilde{u}|^{\tilde{p}}-|u|^{\tilde{p}}\right)\right)(\tilde{p}-1)|\tilde{u}|^{\tilde{p}-2}f,\vspace{0.2cm}\\
 &  & L_{13}=\left(|x|^{-1}\ast\tilde{p}\left(|\tilde{u}|^{\tilde{p}-2}\tilde{u}-|u|^{\tilde{p}-2}u\right)f\right)|u|^{\tilde{p}-2}u,\vspace{0.2cm}
\end{eqnarray*}
and
\[
L_{14}=\left(|x|^{-1}\ast\tilde{p}|\tilde{u}|^{\tilde{p}-2}\tilde{u}f\right)\left(|\tilde{u}|^{\tilde{p}-2}u-|u|^{\tilde{p}-2}u\right),
\]
respectively. We will estimate $L_{11},L_{12}$ for instance and leave
the estimates of $L_{13},L_{14}$ for the interested readers. We assume
that $\tilde{p}>2$, for otherwise $L_{11}\equiv0$, we are done.
Note that $\left||u|^{\tilde{p}-2}-|\tilde{u}|^{\tilde{p}-2}\right|\le|u-\tilde{u}|^{\tilde{p}-2}$
since $0<\tilde{p}-2<1$. Set $\tilde{q}=6(2\tilde{p}-1)/7$. Thus
(\ref{eq: A3}) gives that
\[
\|\left(|\tilde{u}|^{\tilde{p}-2}-|u|^{\tilde{p}-2}\right)f\|_{\frac{\tilde{q}}{\tilde{p}-1}}\le\|u-\tilde{u}\|_{\tilde{q}}^{\tilde{p}-2}\|f\|_{\tilde{q}}\le\de^{\tilde{p}-2}\|f\|_{\X}.
\]
Combining above inequality together with (\ref{eq: A2}) yields that
\begin{equation}
\|L_{11}\|_{2}\apprle_{p,\|u\|_{\X}}\de^{\tilde{p}-2}\|f\|_{\X}.\label{eq: A-estimate of J-11}
\end{equation}
Here we used the assumption $\|u-\tilde{u}\|_{\X}+|\tilde{p}-p|<\de$,
which implies that $\|\tilde{u}\|_{\X}\le\|u\|_{\X}+1$ and $\tilde{p}\le p+1$.
To estimate $L_{12}$, note that $\left||u|^{\tilde{p}}-|\tilde{u}|^{\tilde{p}}\right|\apprle_{p}\left(|u|^{\tilde{p}-1}+|\tilde{u}|^{\tilde{p}-1}\right)|u-\tilde{u}|$.
Then (\ref{eq: A2}) implies that
\[
\left\Vert \left(|x|^{-1}\ast\left(|u|^{\tilde{p}}-|\tilde{u}|^{\tilde{p}}\right)\right)\right\Vert _{\tilde{r}}\apprle_{p,\|u\|_{\X}}\|u-\tilde{u}\|_{q}\apprle_{p,\|u\|_{\X}}\de
\]
since $\|u-\tilde{u}\|_{q}\le\|u-\tilde{u}\|_{\X}\le\de$, where $r$
is given by $1/\tilde{r}+1=1/3+\tilde{p}/\tilde{q}$ with $\tilde{q}=6(2\tilde{p}-1)/7$.
Thus combining above inequality together with (\ref{eq: A3}) gives
us that
\begin{equation}
\|L_{12}\|_{2}\apprle_{p,\|u\|_{\X}}\de\|f\|_{\X}.\label{eq: A.estimate of J-12}
\end{equation}
Since we can obtain similar estimates for $L_{13},L_{14}$ as above,
it becomes obvious from e.g. (\ref{eq: A-estimate of J-11}) and (\ref{eq: A.estimate of J-12})
that we can choose $\de>0$ sufficiently small such that (\ref{eq: A.esti-J-1})
holds. This gives the estimate of $L_{1}$.

Next we estimate $L_{2}$. We have $L_{2}=\sum_{i=1}^{4}L_{2i},$
where
\begin{eqnarray*}
 &  & L_{21}=\left(|x|^{-1}\ast|u|^{p}\right)\left((\tilde{p}-1)|u|^{\tilde{p}-2}-(p-1)|u|^{p-2}\right)f,\vspace{0.5cm}\\
 &  & L_{22}=\left(|x|^{-1}\ast\left(|u|^{\tilde{p}}-|u|^{p}\right)\right)(\tilde{p}-1)|u|^{\tilde{p}-2}f,\vspace{0.2cm}\\
 &  & L_{23}=\left(|x|^{-1}\ast\left(p|u|^{p-2}uf\right)\right)\left(|u|^{\tilde{p}-2}u-|u|^{p-2}u\right),\vspace{0.2cm}
\end{eqnarray*}
and
\[
L_{24}=\left(|x|^{-1}\ast\left(\tilde{p}|u|^{\tilde{p}-2}u-p|u|^{p-2}u\right)f\right)|u|^{\tilde{p}-2}u
\]
 respectively. We estimate the first term $L_{21}$ for instance,
and leave the estimates of $L_{22},$ $L_{23}$ and $L_{24}$ for
the interested readers. Note that
\begin{equation}
\left|(p-1)|u|^{p-2}-(\tilde{p}-1)|u|^{\tilde{p}-2}\right|\le|p-\tilde{p}||u|^{p-2}+(\tilde{p}-1)\left||u|^{p-2}-|u|^{\tilde{p}-2}\right|.\label{eq: A. esti-J-21-1}
\end{equation}
Set $q=6(2p-1)/7$. Then by (\ref{eq: A2}), (\ref{eq: A3}) and above
inequality we deduce that
\[
\|L_{21}\|_{2}\apprle_{p,\|u\|_{\X}}\left(|p-\tilde{p}|+\left\Vert |u|^{p-2}-|u|^{\tilde{p}-2}\right\Vert _{\frac{q}{p-2}}\right)\|f\|_{\X}.
\]
We estimate the second term in the bracket of above inequality as
follows. We only consider the case $\tilde{p}>p$. The case $\tilde{p}<p$
can be considered similarly. By an elementary calculation, we obtain
that
\begin{eqnarray*}
\left||u|^{p-2}-|u|^{\tilde{p}-2}\right|\apprle_{p,\tilde{p},M}|p-\tilde{p}||u|^{p-2} &  & \text{on }\{1/M\le|u|\le M\},
\end{eqnarray*}
for any given constant $M>1$, and
\begin{eqnarray*}
\left||u|^{p-2}-|u|^{\tilde{p}-2}\right|\le2|u|^{\tilde{p}-2}\chi_{\{|u|>M\}} &  & \text{on }\{|u|>M\},
\end{eqnarray*}
where $\chi_{\{|u|>M\}}(x)=1$ if $|u(x)|>M$ and $\chi_{\{|u|>M\}}(x)=0$
if $|u(x)|\le M$, and
\begin{eqnarray*}
\left||u|^{p-2}-|u|^{\tilde{p}-2}\right|\le2|u|^{p-2}\chi_{\{|u|<1/M\}} &  & \text{on }\{|u|<1/M\}.
\end{eqnarray*}
Thus we can obtain that
\[
\begin{aligned}\left\Vert |u|^{p-2}-|u|^{\tilde{p}-2}\right\Vert _{\frac{q}{p-2}} & \le C\left(M,p,\|u\|_{\X}\right)|p-\tilde{p}|+C_{p}\left(\int_{\R^{3}}|u|^{\frac{\tilde{p}-2}{p-2}q}\chi_{\{|u|>M\}}\D x\right)^{(p-2)/q}\\
 & \quad+\left\Vert u\chi_{\{|u|<1/M\}}\right\Vert _{\X}^{p-2}.
\end{aligned}
\]
 It is elementary to show that
\begin{eqnarray*}
\left(\int_{\R^{3}}|u|^{\frac{\tilde{p}-2}{p-2}q}\chi_{\{|u|>M\}}\D x\right)^{(p-2)/q}\to0 &  & \text{as }M\to\wq\,\text{and }\tilde{p}\to p
\end{eqnarray*}
and that
\begin{eqnarray*}
\left\Vert u\chi_{\{|u|<1/M\}}\right\Vert _{q}^{p-2}\to0 &  & \text{as }M\to\wq.
\end{eqnarray*}
Thus for given $\ep>0$, we first choose $M>1$ sufficiently large
such that
\[
\left(\int_{\R^{3}}|u|^{\frac{\tilde{p}-2}{p-2}q}\chi_{\{|u|>M\}}\D x\right)^{(p-2)/q}+\left\Vert u\chi_{\{|u|<1/M\}}\right\Vert _{\X}^{p-2}\apprle_{p}\ep,
\]
and then fix such $M$ and choose $\de>0$ sufficiently small enough
such that $C\left(M,\|u\|_{\X},p\right)|p-\tilde{p}|\apprle_{p}\ep$.
This proves that
\begin{equation}
\|L_{21}\|_{2}\apprle_{p,\|u\|_{\X}}\ep\|f\|_{\X}.\label{eq: A.esti-J-21}
\end{equation}
Since we can derive above estimates for $L_{22},$ $L_{23}$ and $L_{24}$,
we conclude that (\ref{eq: A.esti-J-2}) holds. This gives the estimate
for $L_{2}$.

Finally, combining (\ref{eq: B. Sobolev ineq}), (\ref{eq: A.esti. of J-1-1}),
(\ref{eq: A.esti-J-1}) and (\ref{eq: A.esti-J-2}) gives us the estimate
(\ref{eq: A. estimate of J-1}), and thus follows the continuity of
$\pa_{u}F_{1}$. As we can prove similarly the continuity of the derivatives
$\pa_{\la}F_{1},\pa_{p}F_{1}$, the proof of Lemma \ref{lem: Regularity of F}
is complete.
\end{proof}
\emph{Acknowledgment. }
The author is financially supported by the Academy of Finland, project
259224. He would like to thank the postdoc researcher Zhuomin Liu
in the Department of Mathematics and Statistics of the University
of Jyv\"askyl\"a for useful discussions.

\end{document}